\documentclass[12pt,english]{amsart}
\usepackage[T1]{fontenc}
\usepackage[latin9]{inputenc}
\usepackage{geometry}
\usepackage{amsthm}
\usepackage{amssymb}
\usepackage{babel}

\makeatletter
\numberwithin{equation}{section}
\numberwithin{figure}{section}
\theoremstyle{plain}
\newtheorem{thm}{\protect\theoremname}[section]
  \theoremstyle{plain}
  \newtheorem{conjecture}[thm]{\protect\conjecturename}
  \theoremstyle{plain}
  \newtheorem{lem}[thm]{\protect\lemmaname}
  \theoremstyle{definition}
  \newtheorem{defn}[thm]{\protect\definitionname}
  \theoremstyle{plain}
  \newtheorem{prop}[thm]{\protect\propositionname}
  \theoremstyle{remark}
  \newtheorem{rem}[thm]{\protect\remarkname}
  \theoremstyle{plain}
  \newtheorem{cor}[thm]{\protect\corollaryname}

\newcommand{\dlim}{\mathop{\underrightarrow{\rm lim}}\nolimits}
\newcommand{\Rad}{\mathop{\rm Rad}\nolimits}
\newcommand{\en}{\mathop{\rm End}\nolimits}

\newcommand{\rk}{\mathop{\rm rk}\nolimits}

\makeatother

  \providecommand{\conjecturename}{Conjecture}
  \providecommand{\corollaryname}{Corollary}
  \providecommand{\definitionname}{Definition}
  \providecommand{\lemmaname}{Lemma}
  \providecommand{\propositionname}{Proposition}
  \providecommand{\remarkname}{Remark}
\providecommand{\theoremname}{Theorem}

\begin{document}

\title{Inner Ideals of Simple Locally Finite Lie Algebras}

\author{A.A. Baranov}

\address{Department of Mathematics\\
University of Leicester\\
 Leicester, LE1 7RH, UK}

\email{ab155@le.ac.uk}

\author{J. Rowley}

\address{Department of Mathematics\\
University of Leicester\\
 Leicester, LE1 7RH, UK}

\email{jrdr1@le.ac.uk}
\begin{abstract}
Inner ideals of simple locally finite dimensional Lie algebras over
an algebraically closed field of characteristic 0 are described. In
particular, it is shown that a simple locally finite dimensional Lie
algebra has a non-zero proper inner ideal if and only if it is of
diagonal type. Regular inner ideals of diagonal type Lie algebras
are characterized in terms of left and right ideals of the enveloping
algebra. Regular inner ideals of finitary simple Lie algebras are
described.
\end{abstract}
\maketitle

\section{Introduction}

An inner ideal of a Lie algebra $L$ is a subspace $I$ of $L$ such
that $[I[I,L]]\subseteq I$. Inner ideals were first systematically
studied by Benkart \cite{bib:Benk1,bib:Benk2} and proved to be useful
in classifying simple Lie algebras, both of finite and infinite dimension.
They play a role similar to one-sided ideals of associative algebras
in developing Artinian structure theory for Lie algebras \cite{bib:FGG}.
They are also useful in constructing gradings of Lie algebras \cite{FGGN}. 

Throughout the paper, the ground field $F$ is assumed to be algebraically closed
of characteristic zero. 
In the paper we study inner ideals of simple locally finite Lie algebras
over $F$. Recall
that an algebra is called \emph{locally finite} if every finitely
generated subalgebra is finite dimensional. All locally finite algebras
will be considered to be infinite dimensional. Although the full classification
of simple locally finite Lie algebras seems to be impossible to obtain,
there are two classes of these algebras which have especially nice
properties and can be characterized in many different ways. Those
are finitary simple Lie algebras and diagonal simple locally finite
Lie algebras. Recall that an infinite dimensional Lie algebra is called
\emph{finitary} if it consists of finite-rank linear transformations
of a vector space. It is easy to see that finitary Lie algebras are
locally finite. Diagonal locally finite Lie algebras were introduced
in \cite{bib:Bar3} and are defined as limits of ``diagonal'' embeddings
of finite dimensional Lie algebras (see Definition \ref{diagonal}
for details). They can be also characterized as Lie subalgebras of
locally finite associative algebras \cite[Corollary 3.9]{bib:Bar5}.
In Section \ref{sec:non-diag} we prove the following theorem, which
is one of our main results.
\begin{thm}
\label{Mainresult} Let $F$ be an algebraically closed field of 
characteristic zero. A simple locally finite Lie algebra over $F$
has a proper non-zero inner ideal if and only if it is diagonal.
\end{thm}
The theorem shows that for locally finite simple Lie algebras, non-trivial inner ideals appear only in diagonal
Lie algebras and gives another characterization of this class of algebras.
The complete classification of diagonal simple locally finite Lie
algebras was obtained in \cite{bib:BBZ} and we need some notation
to state it here. 

Let $A$ be an associative enveloping algebra of a Lie algebra $L$
(i.e. $L$ is a Lie subalgebra of $A$ and $A$ is generated by $L$
as an associative algebra). We say that $A$ is a \emph{$\mathfrak{P}$-enveloping
algebra} of $L$ if $[A,A]=L$. Assume now that $A$ has an involution
(which will be always denoted by $\ast$). Then the set $\mathfrak{u}^{\ast}(A)=\{a\in A\mid a^{\ast}=-a\}$
of skew symmetric elements of $A$ is a Lie subalgebra of $A$. Let
$\mathfrak{su}^{\ast}(A)=[\mathfrak{u}^{\ast}(A),\mathfrak{u}^{\ast}(A)]$
denote the commutator subalgebra of $\mathfrak{u}^{\ast}(A)$. We
say that $A$ is a \emph{$\mathfrak{P}^{*}$-enveloping algebra} of
$L$ if $\mathfrak{su}^{\ast}(A)=L$. It is shown in \cite[1.3-1.6]{bib:BBZ}
that every simple diagonal locally finite Lie algebra $L$ has a unique
involution simple \emph{$\mathfrak{P}^{*}$-enveloping algebra} $A(L)$
(which is necessarily locally finite). Moreover, the mapping $L\mapsto A(L)$
is a bijective correspondence between the set of all (up to isomorphism)
infinite dimensional simple diagonal locally finite Lie algebras and
the set of all (up to isomorphism) infinite dimensional involution
simple locally finite associative algebras (the inverse map is $A\mapsto\mathfrak{su}^{\ast}(A)$).
Similarly, every simple plain (see Definition \ref{diagonal}) locally
finite Lie algebra $L$ has a unique (up to isomorphism and antiisomorphism)
simple\emph{ $\mathfrak{P}$-enveloping algebra} $A(L)$ (which is
necessarily locally finite). Moreover, the mapping $L\mapsto A(L)$
is a bijective correspondence between the set of all (up to isomorphism)
infinite dimensional simple plain locally finite Lie algebras and
the set of all (up to isomorphism and antiisomorphism) infinite dimensional
simple locally finite associative algebras (the inverse map is $A\mapsto[A,A]$). 

In Section \ref{sec:diag} we introduce and describe basic properties
of so-called regular inner ideals of simple diagonal locally finite
Lie algebras, see Definition \ref{regular} and Propositions \ref{RLinner}
and \ref{RLinner*}. They are induced by left and right ideals of
the $\mathfrak{P}$- (and $\mathfrak{P}^{*}$)-enveloping algebras.
We believe that the following conjecture is true.
\begin{conjecture}
Let $L$ be a simple diagonal locally finite Lie algebra over $F$.
Assume that $L$ is not finitary orthogonal. Then every inner ideal
of $L$ is regular.
\end{conjecture}
In Section \ref{sec:diag} we prove some partial results towards the
conjecture (see Theorem \ref{general}) and show that the conjecture
holds in the case of locally semisimple diagonal Lie algebras (see
Corollary \ref{LocSemi}).

In the last section we apply our results to the finitary simple Lie
algebras. These algebras were classified in \cite{bib:Finitary0}.
In particular, there are just three finitary simple Lie algebras over
$F$ of infinite countable dimension: $\mathfrak{sl}_{\infty}(F)$,
$\mathfrak{so}_{\infty}(F)$ and $\mathfrak{sp}_{\infty}(F)$. Since
finitary simple Lie algebras are both diagonal and locally semisimple,
by Corollary \ref{LocSemi}, all their inner ideals are regular, except
in the finitary orthogonal case. The classification of inner ideals
of finitary simple Lie algebras was first obtained by Fern\'{a}ndez
L\'{o}pez, Garc\'{i}a and G\'{o}mez Lozano \cite{bib:Inner} (over
arbitrary fields of characteristic zero), with Benkart and Fern\'{a}ndez
L\'{o}pez \cite{bib:Benk3} settling later the missing case for orthogonal
algebras. We provide an alternative proof for the case of special
linear and symplectic algebras over an algebraically closed field
of characteristic zero (see Theorem \ref{finitReg}). In the case
of orthogonal algebras we describe only regular inner ideals. 

It follows from a general result, proved for nondegenerate Lie algebras
by Draper, Fern\'{a}ndez L\'{o}pez, Garc\'{i}a and G\'{o}mez Lozano,
that a simple locally finite Lie algebra contains proper minimal inner
ideals if and only if it is finitary (see \cite[Theorems 5.1 and 5.3]{bib:DLGL}).
We prove a version of this result for regular inner ideals, see Corollary
\ref{min-reg}. 

We are grateful to Antonio Fern\'{a}ndez L\'{o}pez for 
attracting our interest to inner ideals and useful comments and suggestions.

\section{Preliminaries}

\label{sec:prelim} Recall that a Lie algebra $L$ is called 
\emph{perfect} if $[L,L]=L$. Similarly, an associative algebra $A$ is 
\emph{perfect} if $AA=A$ (which is always true if $A$ contains 
an identity element). Let $L$ be a perfect finite-dimensional Lie algebra. Then
its solvable radical $\Rad L$ annihilates every simple $L$-module
and $L/\Rad L\cong Q_{1}\oplus\dots\oplus Q_{n}$ is the sum of simple
components $Q_{i}$. Denote by $V_{i}$ the first fundamental $Q_{i}$-module
(so $V_{i}$ is natural and $Q_{i}\cong\mathfrak{sl}(V_{i}),\ \mathfrak{so}(V_{i}),\ \mathfrak{sp}(V_{i})$
if $Q_{i}$ is of classical type). The modules $V_{i}$ can be considered
as $L$-modules in an obvious way and are called the \emph{natural}
$L$-modules. Assume that all $Q_{i}$ are of classical type. An $L$-module
$V$ is called \emph{diagonal} if each non-trivial composition factor
of $V$ is a natural or co-natural module (i.e. dual to natural) of
$L$. Otherwise $V$ is called \emph{non-diagonal}. A diagonal $L$-module
$V$ is called \emph{plain} if all $Q_{i}$ are of type $A$ and each
non-trivial composition factor of $V$ is a natural $L$-module. Let
$L'$ be another perfect finite dimensional Lie algebra containing
$L$. If $W$ is an $L'$-module we denote by $W\downarrow L$ the
module $W$ restricted to $L$. Let $V_{1}',\dots,V_{k}'$ be the
natural $L'$-modules. The embedding $L\subseteq L'$ is called \emph{diagonal}
(respectively \emph{plain}) if $(V_{1}'\oplus\dots\oplus V_{k}')\downarrow L$
is a diagonal (respectively plain) $L$-module. By the \emph{rank}
of a perfect finite dimensional Lie algebra we mean the smallest rank
of the simple components of $L/\Rad L$.

We will frequently use the following lemma from \cite{bib:BBZ}.
\begin{lem}
\label{BBZ2.5}\cite[Lemma 2.5]{bib:BBZ} Let $L_{1}\subseteq L_{2}\subseteq L_{3}$
be three perfect finite dimensional Lie algebras. Suppose that the
ranks of $L_{1}$ and $L_{3}$ are greater than 10 and the embedding
$L_{1}\subseteq L_{3}$ is diagonal. Then the embedding $L_{1}\subseteq L_{2}$
is diagonal. Moreover, if the restriction of each natural $L_{2}$-module
to $L_{1}$ is non-trivial then both embeddings $L_{1}\subseteq L_{2}$
and $L_{2}\subseteq L_{3}$ are diagonal.
\end{lem}
We will also use the following obvious property of perfect finite
dimensional Lie algebras. 
\begin{lem}
\label{MaxIdeals} Let $L$ be a perfect finite dimensional Lie algebra
and let $Q_{1},\dots,Q_{n}$ be the simple components of $L/\Rad L$.
Then $L$ has exactly $n$ maximal ideals $M_{1},\dots,M_{n}$ and
$L/M_{i}\cong Q_{i}$. \end{lem}
\begin{proof}
It is suffices to note that any maximal ideal of a perfect Lie algebra
contains its solvable radical.\end{proof}
\begin{defn}
A system of finite dimensional subalgebras $\mathfrak{L}=(L_{\alpha})_{\alpha\in\Gamma}$
of a Lie (or associative) algebra $L$ is called a \emph{local system}
for $L$ if the following are satisfied: 

(1) $L=\bigcup_{\alpha\in\Gamma}L_{\alpha}$ 

(2) for $\alpha,\beta\in\Gamma$ there exists $\gamma\in\Gamma$ such
that $L_{\alpha},L_{\beta}\subseteq L_{\gamma}$. 
\end{defn}
Put $\alpha\leq\beta$ if $L_{\alpha}\subseteq L_{\beta}$. Then $\Gamma$
is a directed set and $L=\dlim L_{\alpha}$. We say that a local system
is \emph{perfect} (resp. \emph{semisimple}) if it consists of perfect
(resp. semisimple) subalgebras. 
\begin{defn}
\label{diagonal} A perfect local system $(L_{\alpha})_{\alpha\in\Gamma}$
is called \emph{diagonal} (resp. \emph{plain}) if for all $\alpha\le\beta$
the embedding $L_{\alpha}\subseteq L_{\beta}$ is diagonal (resp.
plain). A simple locally finite Lie algebra $L$ is called \emph{diagonal}
(resp. \emph{plain}) if it has a diagonal (resp. plain) local system.
Otherwise, $L$ is called \emph{non-diagonal}. 
\end{defn}
Note that plain locally finite Lie algebras are diagonal. 
\begin{lem}
\label{radical} \cite[Theorem 3.2 and Lemma 3]{bib:BS} Let $L$
be a simple locally finite Lie (or associative) algebra. Then $L$
has a perfect local system and if $(L_{\alpha})_{\alpha\in\Gamma}$
is a perfect local system for $L$ then for every $\alpha\in\Gamma$
there exists $\alpha'\in\Gamma$ such that for all $\beta\ge\alpha'$
one has $\Rad L_{\beta}\cap L_{\alpha}=0$. \end{lem}
\begin{proof}
In the case of Lie algebras this was proved in \cite[Theorem 3.2 and Lemma 3]{bib:BS}.
Proof of the associative case is similar.\end{proof}
\begin{defn}
\label{defcon} A perfect local system $(L_{\alpha})_{\alpha\in\Gamma}$
is called\emph{ conical} if $\Gamma$ contains a minimal element
$1$ such that 

(1) $L_{1}\subseteq L_{\alpha}$ for all $\alpha\in\Gamma$; 

(2) $L_{1}$ is simple; 

(3) for each $\alpha\in\Gamma$ the restriction of any natural $L_{\alpha}$-module
to $L_{1}$ has a non-trivial composition factor. 
\end{defn}
By the \emph{rank} of a conical system we mean the rank of the simple
Lie algebra $L_{1}$. Note that property (3) of the definition implies
that for every $\alpha\in\Gamma$ and every simple component $S$
of a Levi subalgebra of $L_{\alpha}$ one has $\rk S\ge\rk L_{1}$.
In particular, all these simple components are classical if $\rk L_{1}\ge9$.
\begin{prop}
\cite[Proposition 3.1]{bib:BBZ} \label{conical} Let $L$ be a simple
locally finite Lie algebra and let $\mathfrak{L}=(L_{\alpha})_{\alpha\in\Gamma}$
be a perfect local system of $L$. Let $Q$ be a finite dimensional
simple subalgebra of $L$. Fix any $\beta\in\Gamma$ such that $Q\subseteq L_{\beta}$.
For $\gamma\ge\beta$, denote by $L_{\gamma}^{Q}$ the ideal of $L_{\gamma}$
generated by $Q$. Put $L_{1}^{Q}=Q$ and $\Gamma^{Q}=\{\gamma\in\Gamma\mid\gamma\ge\beta\}\cup\{1\}$.
Then $\mathfrak{L}^{Q}=(L_{\alpha}^{Q})_{\alpha\in\Gamma^{Q}}$ is
a conical local system of $L$ and the following hold. 

(1) Every natural $L_{\alpha}^{Q}$-module is the restriction of a
natural $L_{\alpha}$-module. In particular, the embedding $L_{\alpha}^{Q}\subseteq L_{\alpha}$
is diagonal.

(2) If the local system $\mathfrak{L}$ is diagonal (resp. plain)
then the local system $\mathfrak{L}^{Q}$ is diagonal (resp. plain). 

(3) If the local system $\mathfrak{L}$ is semisimple then the local
system $\mathfrak{L}^{Q}$ is semisimple. \end{prop}
\begin{proof}
Parts (1) and (2) were proved in \cite{bib:BBZ}. Part (3) is obvious. \end{proof}
\begin{prop}
\cite[Corollary 3.3]{bib:BBZ} Simple locally finite Lie algebras
have conical local systems of arbitrary large rank. \end{prop}
\begin{rem}
Similar results hold for locally finite associative algebras. In particular,
every (involution) simple locally finite associative algebra $A$
has a conical ($*$-invariant) local system of subalgebras, see \cite[Proposition 2.9]{bib:BBZ}.
Moreover, this system will be semisimple if $A$ is locally semisimple. 
\end{rem}
The following two results were essentially proved in \cite[Corollary 3.4]{bib:Bar5}.
\begin{thm}
\label{SimpleCriterion} Let $L$ be a simple locally finite Lie algebra
and let $(L_{\alpha})_{\alpha\in\Gamma}$ be a conical local system
for $L$. Then for every $\alpha\in\Gamma$ there is $\alpha'\in\Gamma$
such that for all $\beta\ge\alpha'$ and all maximal ideals $M$ of
$L_{\beta}$ one has $L_{\alpha}\cap M=0$. In particular, for every
simple component $Q$ of $L_{\beta}/\Rad L_{\beta}$ one has $\dim Q\ge\dim L_{\alpha}$. \end{thm}
\begin{proof}
For each $\gamma\in\Gamma$ we denote by $R_{\gamma}$ the solvable
radical of $L_{\gamma}$, by $S_{\gamma}$ the semisimple quotient
$L_{\gamma}/R_{\gamma}$ and by $S_{\gamma}^{1},\dots,S_{\gamma}^{k_{\gamma}}$
the simple components of $S_{\gamma}$. In particular, $R_{1}=0$
and $L_{1}=S_{1}=S_{1}^{1}$. Fix any $\alpha\in\Gamma$. By Lemma
\ref{radical}, there is $\gamma>\alpha$ such that $R_{\gamma}\cap L_{\alpha}=0$
and by \cite[Corollary 3.4]{bib:Bar5} there is $\alpha'>\gamma$
such that the sets of $S_{1}^{1}-,S_{\gamma}^{1}-,S_{\gamma}^{2}-,\dots,S_{\gamma}^{k_{\gamma}}-$accessible
simple components on level $\beta$ coincide for all $\beta\ge\alpha'$.
Recall that for $\beta>\gamma$, a component $S_{\beta}^{i}$ is $S_{\gamma}^{j}$-accessible
if the restriction of the natural $L_{\beta}$-module $V_{\beta}^{i}$
to $L_{\gamma}$ has a composition factor which is non-trivial as
an $S_{\gamma}^{j}$-module. Fix any $\beta\ge\alpha'$. Let $M$ be
a maximal ideal of $L_{\beta}$. Then by Lemma \ref{MaxIdeals}, $L_{\beta}/M\cong S_{\beta}^{i}$
for some $i.$ More exactly, $M$ is the annihilator of the natural
$L_{\beta}$-module $V_{\beta}^{i}$. Note that all components of
$S_{\beta}$ are $S_{1}^{1}$-accessible by the definition of conical
systems (property (3)). This means that $S_{\beta}^{i}$ is $S_{\gamma}^{j}$-accessible
for all $j$, i.e. all simple components of $S_{\gamma}$ act non-trivially
on $V_{\beta}^{i}$ and cannot be in its annihilator $M$. Therefore
$M\cap L_{\gamma}\subset R_{\gamma}$. Since $R_{\gamma}\cap L_{\alpha}=0$,
one has that $M\cap L_{\alpha}=0$, as required. \end{proof}
\begin{cor}
\label{SimpleCrit} Let $L$ be a simple locally finite Lie algebra
and let $(L_{\alpha})_{\alpha\in\Gamma}$ be a conical local system
for $L$. Then for every finite-dimensional simple subalgebra $Q$
of $L$ there exists $\alpha'\in\Gamma$ such that for all $\beta\ge\alpha'$,
$Q\subseteq L_{\beta}$ and the restriction of every natural $L_{\beta}$-module
$V$ to $Q$ has a non-trivial composition factor, i.e. $\{Q,L_{\beta}\mid\beta\ge\alpha'\}$
is a conical local system of $L$. \end{cor}
\begin{proof}
Fix any $\alpha\in\Gamma$ such that $Q\subseteq L_{\alpha}$. By
Theorem \ref{SimpleCriterion}, there is $\alpha'\in\Gamma$ such
that for all $\beta\ge\alpha'$ and all maximal ideals $M$ of $L_{\beta}$
one has $L_{\alpha}\cap M=0$. Let $V$ be a natural $L_{\beta}$-module.
Then its annihilator $M$ is a maximal ideal of $L_{\beta}$. Since
$Q\cap M=0$, $Q$ acts non-trivially on $V$. 
\end{proof}
We will need a version of the above theorem for associative algebras.
\begin{thm}
\label{SimpleCriterionAssoc} Let $A$ be a (involution) simple locally
finite associative algebra and let $(A_{\alpha})_{\alpha\in\Gamma}$
be a conical perfect ($*$-invariant) local system for $A$. Then
for every $\alpha\in\Gamma$ there is $\alpha'\in\Gamma$ such that
for all $\beta\ge\alpha'$ and all ($*$-invariant) maximal ideals
$M$ of $A_{\beta}$ one has $A_{\alpha}\cap M=0$.\end{thm}
\begin{proof}
The proof is similar to that of the previous theorem.\end{proof}
\begin{prop}
\label{diag_sys} Let $L$ be a simple diagonal locally finite Lie
algebra and let $(L_{\alpha})_{\alpha\in\Gamma}$ be a conical local
system of $L$. Then for every $n\in\mathbb{N}$ there is $\alpha'\in\Gamma$
and a simple subalgebra $Q$ of $L$ with $\rk Q>n$ such that $Q\subseteq L_{\beta}$
for all $\beta\ge\alpha'$ and $\{Q,L_{\beta}\mid\beta\ge\alpha'\}$
is a conical diagonal local system of $L$ of rank $>n$. \end{prop}
\begin{proof}
Since $L$ is diagonal, by \cite[Theorem 3.8]{bib:Bar5} $L$ has
a conical diagonal local system $(M_{\delta})_{\delta\in\Delta}$
of rank $>\max\{10,n\}$. Note that $M_{1}$ is simple of rank $>\max\{10,n\}$.
Put $Q=M_{1}$. By Corollary \ref{SimpleCrit}, there is $\alpha'\in\Gamma$
such that $Q\subseteq L_{\alpha'}$ and for all $\beta\ge\alpha'$
the restriction of every natural $L_{\beta}$-module $V$ to $Q$
has a non-trivial composition factor. It remains to prove that the
embeddings $Q\subseteq L_{\beta_{1}}$ and $L_{\beta_{1}}\subseteq L_{\beta_{2}}$
are diagonal for all $\beta_{2}>\beta_{1}\ge\alpha'$. Fix any $\delta\in\Delta$
such that $L_{\beta_{2}}\subseteq M_{\delta}$, so we have a chain
of embeddings
\[
Q=M_{1}\subseteq L_{\beta_{1}}\subseteq L_{\beta_{2}}\subseteq M_{\delta}.
\]
Since $\rk Q>10$ and the embedding $Q\subseteq M_{\delta}$ is diagonal,
by Lemma \ref{BBZ2.5}, the embedding $Q\subseteq L_{\beta_{2}}$
is diagonal. Applying this lemma again to the triple $Q\subseteq L_{\beta_{1}}\subseteq L_{\beta_{2}}$,
we get that the embeddings $Q\subseteq L_{\beta_{1}}$ and $L_{\beta_{1}}\subseteq L_{\beta_{2}}$
are diagonal, as required. \end{proof}
\begin{thm}
\label{finitaryCriterion} Let $L$ be a simple diagonal locally finite
Lie algebra and let $(L_{\alpha})_{\alpha\in\Gamma}$ be a perfect
local system for $L$. Assume that there is $\alpha\in\Gamma$, a
non-zero $x\in L_{\alpha}$ and a natural number $k$ such that for
all $\beta\ge\alpha$, the rank of $x$ is $\le k$ on every natural
$L_{\beta}$-module. Then $L$ is finitary. \end{thm}
\begin{proof}
By Proposition \ref{diag_sys}, we can assume that $(L_{\alpha})_{\alpha\in\Gamma}$
is a conical diagonal local system for $L$ of rank $>10$. Let $A$
be its involution simple associative $\mathfrak{P}^{*}$-envelope
and let $A_{\alpha}$ be the subalgebra of $A$ generated by $L_{\alpha}$.
Then it follows from the construction of $A$ (see proof of Theorem
1.3 in \cite{bib:BBZ}), that $(A_{\alpha})_{\alpha\in\Gamma}$ is
a conical diagonal local system for $A$, $\mathfrak{su}^{\ast}(A_{\alpha})=L_{\alpha}$,
every natural $L_{\alpha}$-module is lifted to $A_{\alpha}$ and
every irreducible $A_{\alpha}$-module is either natural or co-natural
$L_{\alpha}$-module. Let $B$ be the ideal of $A$ generated by $x$.
Since $x^{*}=-x$, $B$ is $*$-invariant, so $B=A$. Note that $xAx\ne0$.
Indeed, otherwise $A=A^{3}=BAB=0$. Therefore $x$ acts nontrivially
on the left $A$- (and $L$-) module $V=Ax$. We claim that $\dim xAx\le2k^{2}$.
It is enough to show that $\dim xA_{\beta}x\le2k^{2}$ for all large
$\beta$. By Theorem \ref{SimpleCriterionAssoc}, there is $\gamma>\beta$
and a maximal $*$-invariant ideal $M$ of $A_{\gamma}$ such that
$M\cap A_{\beta}=0$. Note that the quotient $Q=A_{\gamma}/M$ is
either simple or the direct sum of two simple components, so $Q$
is isomorphic to $\en U$ or $\en W_{1}\oplus\en W_{2}$ where $U$
and $W_{1}$ are natural $L_{\gamma}$-modules and $W_{2}$ is co-natural.
Since $M\cap A_{\beta}=0$, we have an isomorphic image of $A_{\beta}$
in $Q$. Assume first that $Q\cong\en U$. Since $x$ is of rank $\le k$
on $U$, it is easy to see that $\dim xQx\le k^{2}$ (e.g. by using
the Jordan canonical form of $x$). Similarly, if $Q\cong\en W_{1}\oplus\en W_{2}$,
we get that $\dim xQx\le2k^{2}$. Therefore, $\dim xA_{\beta}x\le2k^{2}$
and $\dim xAx\le2k^{2}$, as required. Thus, $x$ is a finite rank
transformation of $V=Ax$. Note that all finite rank transformations
of $V$ in $L$ form an ideal of $L$. Since $L$ is simple, $V$
is a non-trivial finitary module for $L$, so $L$ is finitary. \end{proof}
\begin{defn}
Let $L$ be a Lie algebra. An \emph{inner ideal} of $L$ is a subspace
$I$ of $L$ such that $[I,[I,L]]\subseteq I$. 
\end{defn}
Although inner ideals are not ideals in general (not even subalgebras)
it is easy to see that they are well-behaved with respect to subalgebras
and factor algebras:
\begin{lem}
\label{Inner0} Let $I$ be an inner ideal of a Lie algebra $L$. 

(1) Let $H$ be a subalgebra of $L$. Then $I\cap H$ is an inner
ideal of $H$. 

(2) Let $J$ be an ideal of L then $(I+J)/J$ is an inner ideal of
$L/J$. 
\end{lem}
The following classifies the inner ideals of the classical finite
dimensional Lie algebras over $F$. This is only a very particular
case of the results proven in \cite{bib:Benk1,bib:Benk3}. 
\begin{thm}
\label{ClassInnerFD}\cite[Theorem 5.1]{bib:Benk1}\cite[Theorem 6.3(i)]{bib:Benk3}
Let $V$ be a finite dimensional vector space over an algebraically
closed field $F$ of characteristic zero. Let $A=\en V$ and $\Phi$
(resp. $\Psi)$ be a non-degenerate symmetric (resp. skew-symmetric)
form on $V$. Let $\ast$ be the involution of $A$ induced by 
$\Phi$ (resp. $\Psi)$.

(1) Let $L=\mathfrak{sl}(V)$. A subspace $I$ of $L$ is a proper inner
ideal of $L$ if and only if there exist idempotents $e$ and $f$
in $A$ such that $I=eAf$ and $fe=0$. 

(2) Let $\mbox{ }L=\mathfrak{sp}(V,\Psi)$ and $\dim V>4$. A subspace
$I$ of $L$ is a proper inner ideal of $L$ if and only if there exists
an idempotent $e$ in $A$ such that $I=eLe^{\ast}$ and $e^{\ast}e=0$
(equivalently, $I=[U,U]=span\{u^{\ast}v+v^{\ast}u|u,v$ $\in U\}$
where $U$ is a totally isotropic subspace of $V$ and $u^{*}v\in\en V$
is defined as $(u^{*}v)(w)=\Psi(w,u)v$ for all $w\in V$). 

(3) Let $L=\mathfrak{o}(V,\Phi)$ and $\dim V>4$. A subspace $I$
of $L$ is a proper inner ideal of $L$ if and only if one of the following
holds. 

(i) $I=eLe^{\ast}$ where $e\in A$ an idempotent such that $e^{*}e=0$. 

(ii) $I=[v,H^{\perp}]$ where $v\in V$ is a nonzero isotropic vector
of $H$, and $H$ is a hyperbolic plane of $V$ (equivalently, there
is a basis $\{x_{1},\dots,x_{n}\}$ of $V$ such that $I$ is the
$F$-span of the matrix units $e_{1j}-e_{j2}$, $j\ge3$, with respect
to this basis \cite[4.1]{bib:Benk3}). 

(iii) $I$ is a Type 1 point space of dimension greater than 1.
\end{thm}
Recall that a subspace $P$ of a Lie algebra $L$ is called a \emph{point
space }if $[P,P]=0$ and $ad_{x}^{2}L=Fx$ for every nonzero element
$x\in P$. Moreover, a point subspace $P$ of $\mathfrak{o}(V,\Phi)$
is said to be of \emph{Type 1} if there is a non-zero vector $u$
in the image of every non-zero $a\in P$. 
\begin{lem}
\cite[Lemma 1.13]{bib:Benk2}\label{Abelian} Let $L$ be a finite
dimensional simple Lie algebra and let $I$ be a proper inner ideal
of $L$. Then $[I,I]=0$, i.e. $I$ is abelian. 
\end{lem}
The following two facts are well known, see for example \cite[Proposition 2.3]{bib:Artinian}.
\begin{lem}
\label{InnerDecompLemma} Let $L$ be a finite dimensional simple
Lie algebra and let $I$ be an inner ideal of $L$. Then $[I,[I,L]]=I$. \end{lem}
\begin{proof}
Let $I$ be an inner ideal of $L$. If $I=L$ then this is obviously
true. Assume that $I$ is proper. Then by Lemma \ref{Abelian}, $I$
is abelian. Let $x\in I$. Then $[x,[x,[x,L]]]\subseteq[x,I]=0$,
so $x$ is ad-nilpotent. By the Jacobson-Morozov Theorem, there exist
$y,h\in L$ such that $\{x,y,h\}$ form an $\mathfrak{sl}_{2}$-triple.
Note that $[x,[x,y]]=[x,h]=-2x$, so $x\in[I,[I,L]]$. This implies
$I=[I,[I,L]]$, as required. \end{proof}
\begin{lem}
\label{InnerDecomposition} Let $L$ be a finite
 dimensional semisimple
Lie algebra. Let $Q_{1},\dots,Q_{n}$ be the simple components of
$L$. Let $I$ be an inner ideal of $L$ and $I_{i}=I\cap Q_{i}$.
Then $I=I_{1}\oplus\dots\oplus I_{n}.$\end{lem}
\begin{proof}
Let $\psi_{k}:L\rightarrow Q_{k}$, $\psi_{k}((q_{1},\dots,q_{n}))=q_{k}$,
be the natural projection and let $J_{k}=\psi_{k}(I)$. We need to
show $J_{k}=I_{k}$. Indeed, by Lemma \ref{Inner0}, $J_{k}$ is an
inner ideal of $Q_{k}$. It is clear that $I_{k}\subseteq J_{k}$.
On the other hand, by Lemma \ref{InnerDecompLemma}, 
\[
J_{k}=[J_{k},[J_{k},Q_{k}]]=[I,[I,Q_{k}]]\subseteq I_{k}.
\]
Therefore $I_{k}=J_{k}$ for all $k$, so $I=I_{1}\oplus\dots\oplus I_{n}$. \end{proof}
\begin{prop}
\label{InnerPerfect} Let $L$ be a perfect finite dimensional Lie
algebra and let $I$ be an inner ideal of $L$. Assume that $(I+M)/M=L/M$
for every maximal ideal $M$ of $L$. Then $I=L$
. \end{prop}
\begin{proof}
Without loss of generality we can assume that $I$ is minimal among all inner
ideals of $L$ satisfying this assumption. By \cite[Lemma 1.1(4)]{bib:Benk2},
for every inner ideal $J$ of $L$ the subspace $J^{[3]}=[J,[J,J]]$
is also an inner ideal of $L$. Note that $I^{[3]}$ satisfies the
assumption of the proposition (as $L/M$ is a simple Lie algebra by
Lemma \ref{MaxIdeals}) and is contained in $I$: $I^{[3]}\subseteq[I,[I,L]]\subseteq I$.
Therefore $I^{[3]}=I$. Now 
\[
[L,I]=[L,[I,[I,I]]]\subseteq[I,[I,L]]\subseteq I
\]
so $I$ is an ideal of $L$. Since $I$ is not contained in any maximal
ideal, $I=L$, as required.\end{proof}
\begin{lem}
\label{Inner1}\cite[Lemma 4.23]{bib:Benk1} Let $L$ be a classical
simple finite dimensional Lie algebra and let $V$ be the natural
module for $L$. Let $I$ be a proper inner ideal of $L$. Then $I^{3}V=0$.
In particular, $x^{3}V=0$ for all $x\in I$.\end{lem}
\begin{proof}
This was proved in \cite{bib:Benk1} but also follows from the classification
of inner ideals given in Theorem \ref{ClassInnerFD}. Indeed, referring
to the notation of the theorem, suppose $I=eAf$ or $I=eLe^{\ast}$
as in cases 1, 2 and 3 part (i). Then $fe=0$ or $e^{\ast}e=0$, so
$I^{2}=0$. Now suppose $I=[v,H^{\perp}]$ as in case 3 part (ii).
Then $I$ is the $F$-span of the matrix units $e_{1j}-e_{j2}$, $j\ge3$.
Note that $I^{2}=Fe_{12}$ and $I^{3}=0$. Finally consider 
part (iii) of case 3 of Theorem \ref{ClassInnerFD}. If $I$ is a point space of type 1 then $I$ is a subspace
of $eLe^{*}$ for some idempotent $e$ with $e^{*}e=0$ (see \cite[Proposition 4.3]{bib:Benk3}).
Thus again $I^{2}=0$. 
\end{proof}

\section{Non-diagonal locally finite Lie algebras }

\label{sec:non-diag} The aim of this section is to prove Theorem
\ref{Mainresult}: a simple locally finite Lie algebra over $F$ has
a proper nonzero inner ideal if and only if it is diagonal. First
we are going to show that every simple diagonal locally finite Lie
algebra has a non-zero proper inner ideal. This will be generalized
in the next section were we describe all regular inner ideals of diagonal
Lie algebras. 
\begin{prop}
\label{Prop1} Every simple diagonal locally finite Lie algebra has
a proper non-zero inner ideal. \end{prop}
\begin{proof}
Let $L$ be a simple diagonal locally finite Lie algebra. By \cite[Theorems 1.1 and 1.2]{bib:BBZ}
there exists an involution simple locally finite associative algebra
$A$ such that $L=\mathfrak{su}^{\ast}(A)$. By \cite[Corollary 2.11]{bib:BBZ},
for every integer $m$, $A$ contains an involution simple finite
dimensional subalgebra $A_{1}$ of dimension greater than $m$. It
is well known $A_{1}$ is isomorphic to a matrix algebra $M_{n}(F)$
with orthogonal or symplectic involution or the direct sum of two
copies of $M_{n}(F)$ with involution permuting the components and
$L_{1}=\mathfrak{su}^{\ast}(A_{1})$ is a finite dimensional classical
Lie algebra isomorphic to $\mathfrak{sl}_{n}$, $\mathfrak{sp}_{n}$,
or $\mathfrak{o}_{n}$ (see for example \cite[Lemmas 2.1 and 2.2]{bib:BZ2}).
Fix any idempotent $e$ in $A_{1}$ such that $e^{\ast}e=0$ and $eL_{1}e^{*}\ne0$
(see Theorem \ref{ClassInnerFD}). Put $I=eAe^{\ast}\cap\mathfrak{su}^{\ast}(A)$.
Note that $I^{2}=0$, so $I$ is a proper non-zero subspace of $L$.
Since $(eAe^{*})A(eAe^{*})\subseteq eAe^{*}$, one has $[I,[I,L]]\subseteq I$,
so $I$ is an inner ideal of $L$. \end{proof}
\begin{rem}
As it was mentioned to us by Antonio Fern\'{a}ndez L\'{o}pez, the proposition also
follows from \cite[Corollary 2.3]{bib:FLYG}, because every simple
diagonal locally finite Lie algebra $L$ has an algebraic adjoint
representation \cite[Corollary 3.9(6)]{bib:Bar5}, and hence a non-zero abelian inner ideal. \end{rem}
\begin{lem}
\label{Lemma A} Let $L$ be a non-diagonal simple locally finite
Lie algebra. Let $(L_{\alpha})_{\alpha\in\Gamma}$ be a conical local
system of $L$ of rank $>10$. Then for every $\beta$ there exists
$\beta'\ge\beta$ such that for all $\gamma\ge\beta'$ the embedding
$L_{\beta}\subseteq L_{\gamma}$ is non-diagonal. \end{lem}
\begin{proof}
Let $\beta\in\Gamma$. Suppose to the contrary that for every $\beta'\ge\beta$
there is $\gamma\ge\beta'$ such that the embedding $L_{\beta}\subseteq L_{\gamma}$
is diagonal. Since $L_{\beta}\subseteq L_{\beta'}\subseteq L_{\gamma}$,
by Lemma \ref{BBZ2.5} the embedding $L_{\beta}\subseteq L_{\beta'}$
is diagonal for all $\beta'\ge\beta$. Fix any simple component $Q$
of a Levi subalgebra of $L_{\beta}$. Then $\rk Q>10$ and by Corollary
\ref{SimpleCrit}, there is $\alpha'>\beta$ such that $\mathfrak{L}=\{Q,L_{\gamma}\mid\gamma\ge\alpha'\}$
is a conical local system of $L$. We are going to prove that $\mathfrak{L}$
is a diagonal local system, so $L$ is a diagonal Lie algebra, which
is a contradiction. We already know that all embeddings $Q\subseteq L_{\gamma}$,
$\gamma\ge\alpha'$, are diagonal. Consider any $\xi>\zeta\ge\alpha'$.
Then we have a chain of embeddings $Q\subseteq L_{\zeta}\subseteq L_{\xi}$.
By construction both $Q\subseteq L_{\zeta}$ and $Q\subseteq L_{\xi}$
are diagonal. Since $\mathfrak{L}$ is conical and $\rk Q>10$, by
Lemma \ref{BBZ2.5} the embedding $L_{\zeta}\subseteq L_{\xi}$ is
diagonal, as required. \end{proof}
\begin{lem}
\cite[Lemma 4.5]{bib:Bar3} \label{Levi} Let $L_{1}\subset L_{2}$
be finite dimensional Lie algebras; let $S_{1}$ and $S_{2}$ be Levi
subalgebras of $L_{1}$ and $L_{2}$, respectively. Then there exists
an automorphism $\theta$ of $L_{2}$ such that $\theta(S_{1})\subseteq S_{2}$
and $\theta(l)=l+r(l)$ for all $l\in L_{2}$, with $r(l)$ being
in the nilpotent radical of $L_{2}$. Moreover the monomorphism $S_{1}\subseteq S_{2}$
induced by $\theta$ does not depend on the choice of such $\theta$.
\end{lem}
In what follows we will use the function $\delta$ introduced in \cite{bib:Bar3}.
This is a function defined on the weights (and modules) of simple
Lie algebras. The function $\delta$ takes integral values (and also half-integral
values in the case of algebras of type B). 
Let $L$ be a finite dimensional simple Lie algebra of rank
$m$. Denote by $\omega_{1},\dots,\omega_{m}$ the fundamental weights
of $L$ and by $\alpha_{1},\dots\alpha_{m}$ the simple roots of $L$.
The function $\delta$ is linear on weights and defined by its values on the
fundamental weights. 
In the following list $\delta(\omega_{i})=p_{i}$, $1\le i\le m$, 
is abbreviated to 
$\delta=(p_{1},\dots,p_{m})$. 
$$
\begin{array}{ll}
\delta = (1,2,\ldots,k,k,\ldots,2,1) & (A_{2k}); \\
\delta = (1,2,\ldots,k+1,\ldots,2,1) &    (A_{2k+1}); \\
\delta = (1,2,\ldots,m-2,m-1,m)  &   (C_{m}, m\ge2); \\
\delta = (1,2,\ldots,m-2,m-1,[\frac{m}{2}])  &  (B_{m}, m\ge3); \\
\delta = (1,2,\ldots,2k-2,k-1,k)  &   (D_{2k}, k\ge2); \\
\delta = (1,2,\ldots,2k-1,k,k) &    (D_{2k+1}, k\ge2); \\
\delta = (2,2,3,4,3,2)  &   (E_{6}); \\
\delta = (2,2,3,4,3,2,1) &    (E_{7}); \\
\delta = (4,5,7,10,8,6,4,2) &    (E_{8}); \\
\delta = (2,3,2,1)  &   (F_{4}); \\
\delta = (1,2) &    (G_{2}). 
\end{array}
$$
It is easy to verify that $\delta(\alpha_{i})\ge0$ for all $i=1,\dots, m$.
Let $V$ be a finite dimensional $L$-module and 
$M$ be its set of weights then set
$\delta(V)=\sup\{\delta(\mu)\}_{\mu\in M}$. 
Since the value of
$\delta$ on the simple roots is non-negative this implies that  $\delta(V)=\delta(\mu_{h})$
where $\mu_{h}$ is the highest weight of $V$. 
If rank of $L$ is greater than $10$ then the following holds.
The $L$-module $V$ is trivial if and only if $\delta(V)=0$; $V$ is non-trivial
diagonal if and only if $\delta(V)=1$; $V$ is non-diagonal if and
only if $\delta(V)\ge2$ (see \cite[Section 6]{bib:Bar3} for details). 
\begin{lem}
\label{Lemma B} Let $L_{1}\subseteq L_{2}\subseteq L_{3}$ be three
perfect finite dimensional Lie algebras such that $L_{1}$ is simple
and $\rk L_{1}>10$. Suppose that the embedding $L_{2}\subseteq L_{3}$
is non-diagonal and the restriction of every natural $L_{2}$-module
to $L_{1}$ is non-trivial. Then there is a natural $L_{3}$-module
$V$ such that $\delta(V\downarrow L_{1})>1$. In particular, the
restriction of $V$ to $L_{1}$ is non-diagonal.\end{lem}
\begin{proof}
By using the Levi-Malcev Theorem and Lemma \ref{Levi} we can reduce this
to the case of Levi subalgebras, one embedded into the next, so we
can assume that the $L_{i}$ are semisimple. Since the embedding $L_{2}\subseteq L_{3}$
is non-diagonal, there exists a natural $L_{3}$-module, say, $V$
such that $V\downarrow L_{2}$ has an irreducible component $W$ which
is not trivial, natural, or co-natural. We have
\[
\delta(V\downarrow L_{1})=\delta((V\downarrow L_{2})\downarrow L_{1})\ge\delta(W\downarrow L_{1})
\]
It remains to show that $\delta(W\downarrow L_{1})>1$. The module
$W$ can be represented in the form $W=W_{1}\otimes\dots\otimes W_{k}$
where each $W_{i}$ is a non-trivial irreducible module for a simple
component $S_{i}$ of $L_{2}$. Then we have two cases: either at
least two $W_{i}$ are non-trivial or at least one $W_{i}$ is not trivial,
natural, or co-natural. For the first case, without loss of
generality we may assume that there are just two non-trivial $W_{i}$,
so that $W=W_{1}\otimes W_{2}$. Using \cite[Lemma 7.2]{bib:Bar3}
we see that 
\[
\delta(W\downarrow L_{1})\ge\delta((W\downarrow S_{1})\downarrow L_{1})+\delta((W\downarrow S_{2})\downarrow L_{1})\ge2.
\]
 In the second case we may assume that $W=W_{1}$ where $W_{1}$ is
a non-trivial, non-natural and non-conatural $S_{1}$-module. Then
using \cite[Lemma 6.7]{bib:Bar3}, we get 
\[
\delta(W\downarrow L_{1})\ge\delta((W\downarrow S_{1})\downarrow L_{1})\ge\delta(W_{1}\downarrow L_{1})>\delta(V_{1}\downarrow L_{1})\ge1
\]
where $V_{1}$ is the natural $S_{1}$-module. In both cases $\delta(V\downarrow L_{1})>1$,
so $V$ is a non-diagonal $L_{1}$-module. \end{proof}
\begin{lem}
\label{simple} Let $L$ be a non-diagonal simple locally finite Lie
algebra and let $\mathfrak{L}$ be a conical perfect local system
for $L$ of rank $>10$. Let $n$ be a positive integer and let $S$
be a finite dimensional simple subalgebra of $L$. Then there exists
a chain of subalgebras $M_{1}\subseteq M_{2}\subseteq\dots\subseteq M_{n}$
of $L$ and subalgebras $S_{i}\subseteq M_{i}$, $1\le i\le n$, such
that $M_{1}=S_{1}=S$, for each $i=2,\dots,n$, $M_{i}\in\mathfrak{L}$,
$S_{i}$ is a simple component of a Levi subalgebra of $M_{i}$ and
the restriction $V_{i}\downarrow S_{i-1}$ is a non-diagonal $S_{i-1}$-module
where $V_{i}$ is the natural $M_{i}$-module corresponding to $S_{i}$.
Moreover, $\delta(V_{n}\downarrow S)>n/2$. \end{lem}
\begin{proof}
We construct the algebras $M_{i}$ and $S_{i}$ by induction. Recall
that $M_{1}=S_{1}=S$. Assume that $M_{i-1}$ and $S_{i-1}$ have
been constructed. By Corollary \ref{SimpleCrit}, there is an algebra
$Q_{i}\in\mathfrak{L}$ such that $S_{i-1}\subseteq Q_{i}$ and the
restriction of every natural $Q_{i}$-module to $S_{i-1}$ is non-trivial.
By Lemma \ref{Lemma A}, there is $M_{i}\in\mathfrak{L}$ such that
$Q_{i}\subseteq M_{i}$ and the embedding $Q_{i}\subseteq M_{i}$
is non-diagonal. Therefore by Lemma \ref{Lemma B}, there is a simple
component $S_{i}$ of a Levi subalgebra of $M_{i}$ such that the
restriction $V_{i}\downarrow S_{i-1}$ is a non-diagonal $S_{i-1}$-module
and $\delta(V_{i}\downarrow S_{i-1})>1$ where $V_{i}$ is the natural
$M_{i}$-module corresponding to $S_{i}$. Let $W_{i-1}$ be any non-diagonal
composition factor of the restriction $V_{i}\downarrow S_{i-1}$.
Then $W_{i-1}$ can be viewed as both an $S_{i-1}$- and $M_{i-1}$-module.
Similar to the proof of Lemma \ref{Lemma B}, using \cite[Lemma 6.7]{bib:Bar3},
we get that 
\[
\delta(V_{i}\downarrow S_{1})=\delta((V_{i}\downarrow M_{i-1})\downarrow S_{1})\ge\delta(W_{i-1}\downarrow S_{1})>\delta(V_{i-1}\downarrow S_{1})
\]
Therefore $\delta(V_{n}\downarrow S_{1})>\delta(V_{n-1}\downarrow S_{1})>\dots>\delta(V_{1}\downarrow S_{1})=1$.
Since $\delta$ has half-integer values only, this implies $\delta(V_{n}\downarrow S_{1})>n/2$. \end{proof}
\begin{prop}
\label{simple2} Let $L$ be a non-diagonal simple locally finite
Lie algebra and let $\mathfrak{L}$ be a conical perfect local system
for $L$ of rank $>10$. Let $n$ be a positive integer and let $S$
be a finite dimensional simple subalgebra of $L$. Then there exists
a subalgebra $M\in\mathfrak{L}$ containing $S$ such that for every
$M'\in\mathfrak{L}$ containing $M$ and every natural $M'$-module
$V$ one has $\delta(V\downarrow S)>n$. \end{prop}
\begin{proof}
By Lemma \ref{simple}, there exists $Q_{1}\in\mathfrak{L}$ containing
$S$ and a simple component $S_{1}$ of a Levi subalgebra of $Q_{1}$
such that $\delta(V_{1}\downarrow S)>n$ where $V_{1}$ is the natural
$Q_{1}$-module corresponding to $S_{1}$. By Theorem \ref{SimpleCriterion},
there exists $M\in\mathfrak{L}$ containing $Q_{1}$ such that for
every $M'\in\mathfrak{L}$ containing $M$ and every maximal ideal
$N$ of $M'$ one has $Q_{1}\cap N=0$, so $V\downarrow S_{1}$ is
a non-trivial $S_{1}$-module for every natural $M'$-module $V$.
It remains to show that $\delta(V\downarrow S)>n$. Let $W_{1}$ be
any non-trivial composition factor of the restriction $V\downarrow S_{1}$.
Then $W_{1}$ can be viewed as both $S_{1}$ and $Q_{1}$-module.
Similar to the proof of Lemma \ref{Lemma B}, using \cite[Lemma 6.7]{bib:Bar3},
we get that 
\[
\delta(V\downarrow S)=\delta((V\downarrow Q_{1})\downarrow S)\ge\delta(W_{1}\downarrow S)\ge\delta(V_{1}\downarrow S)>n.
\]
\end{proof}
\begin{prop}
\label{Prop2} Let $L$ be a simple non-diagonal locally finite Lie
algebra. Then $L$ has no non-zero proper inner ideals. \end{prop}
\begin{proof}
Let $(L_{\alpha})_{\alpha\in\Gamma}$ be a perfect conical local system
for $L$ of rank $>10$. Let $R_{\alpha}$ be the solvable radical
of $L_{\alpha}$ and let $S_{\alpha}$ be a Levi subalgebra of $L_{\alpha}$
so that $L_{\alpha}=S_{\alpha}\oplus R_{\alpha}$ for $\alpha\in\Gamma$.
Assume $I$ is a proper non-zero inner ideal of $L$. For $\alpha\in\Gamma$
put $I_{\alpha}=I\cap L_{\alpha}$. Then $I_{\alpha}$ is an inner
ideal of $L_{\alpha}$ by Lemma \ref{Inner0}. Fix $\alpha_{1}\in\Gamma$
such that $I_{\alpha_{1}}$ is a proper non-zero inner ideal of $L_{\alpha_{1}}$.
By Lemma \ref{radical}, there is $\alpha_{2}\ge\alpha_{1}$ such
that $L_{\alpha_{1}}\cap R_{\alpha_{2}}=0$, so $I_{\alpha_{2}}\not\subseteq R_{\alpha_{2}}$
and the image $\overline{I_{\alpha_{2}}}$ of $I_{\alpha_{2}}$ in
the semisimple quotient $\overline{L_{\alpha_{2}}}=L_{\alpha_{2}}/R_{\alpha_{2}}$
is a non-zero inner ideal of $\overline{L_{\alpha_{2}}}\cong S_{\alpha_{2}}$.
It follows from Lemmas \ref{InnerDecomposition} and \ref{Inner1}
that $\overline{I_{\alpha_{2}}}$ contains a non-zero ad-nilpotent element.
Therefore there exist a non-zero ad-nilpotent $s\in S_{\alpha_{2}}$
and an $r\in R_{\alpha_{2}}$ such that $x=s+r\in I_{\alpha_{2}}$.
By the Jacobson-Morozov Theorem, there exists a subalgebra $S$ of
$S_{\alpha_{2}}$ isomorphic to $\mathfrak{sl}_{2}$ containing $s$.
Consider the subalgebra $\hat{S}=S+R_{\alpha_{2}}$ of $L_{\alpha_{2}}$.
Then $\Rad\hat{S}=R_{\alpha_{2}}$ and $I_{0}=I\cap\hat{S}$ is an
inner ideal of $\hat{S}$ containing $x$. By Proposition \ref{simple2},
there exists $\alpha_{3}\in\Gamma$ such that $\hat{S}\subset L_{\alpha_{3}}$
and for every natural $L_{\alpha_{3}}$-module $V$ one has $\delta(V\downarrow S)>2$.
Fix any such module $V$. Note that all composition factors of $V\downarrow\hat{S}$
are irreducible modules for $S$, so $V\downarrow\hat{S}$ has a composition
factor $W$, which is also an irreducible module for $S\cong\mathfrak{sl}_{2}$
with $\delta(W)>2$. It follows from the definition of the function
$\delta$ that $\dim W=\delta(W)+1>3$ (see \cite[Section 6]{bib:Bar3}
for details), so $s^{3}W\ne0$ as $s$ is a basic nilpotent element
of $S$. Since $r\in\Rad\hat{S}=R_{\alpha_{2}}$ and $R_{\alpha_{2}}$
annihilates every composition factor of $V\downarrow L_{\alpha_{2}}$
one has $rW=0$, so 
\[
x^{3}W=(s+r)^{3}W=s^{3}W\ne0.
\]
Therefore, $x^{3}V\ne0$. Let $M$ be the annihilator of $V$ in $L_{\alpha_{3}}$.
Then $M$ is a maximal ideal of $L_{\alpha_{3}}$ and let $a\mapsto\bar{a}$
be the natural homomorphism $L_{\alpha_{3}}\to\overline{L_{\alpha_{3}}}=L_{\alpha_{3}}/M$.
Then $\overline{L_{\alpha_{3}}}$ is a classical simple Lie algebra
of rank $>10$, $\overline{I_{\alpha_{3}}}$ is an inner ideal of
$\overline{L_{\alpha_{3}}}$ and $\overline{x}\in\overline{I_{\alpha_{3}}}$.
Note that $\overline{x}^{3}V=x^{3}V\ne0$. Since $V$ is a natural
module for $\overline{L_{\alpha_{3}}}$, by Lemma \ref{Inner1}, one
has $\overline{I_{\alpha_{3}}}=\overline{L_{\alpha_{3}}}$. Since
this is true for every natural $L_{\alpha_{3}}$-module $V$ (and
so for every maximal ideal $M$ of $L_{\alpha_{3}}$), by Proposition
\ref{InnerPerfect}, $I_{\alpha_{3}}=L_{\alpha_{3}}$. This implies
that $I_{\alpha_{1}}=L_{\alpha_{1}}$, which contradicts the assumption
that $I_{\alpha_{1}}$ is a proper inner ideal of $L_{\alpha_{1}}$.
\end{proof}
\emph{Proof of Theorem \ref{Mainresult}}. This follows from Propositions
\ref{Prop1} and \ref{Prop2}.

\section{Regular Inner Ideals and Diagonal Lie Algebras}

\label{sec:diag} In this section we define regular inner ideals and
discuss inner ideals of simple diagonal locally finite Lie algebras. 
\begin{lem}
\label{reglemma} Let $A$ be an associative algebra and let $L=[A,A]$.
Let $I$ be a subspace of $L$ such that $I^{2}=0.$ Then the following
hold.

(1) $I$ is an inner ideal of $L$ if and only if $ixj+jxi\in I$
for all $i,j\in I$ and all $x\in L$.

(2) $I$ is an inner ideal of $L$ if and only if $iLi\subseteq I$
for all $i\in I$. 

(3) $IAI\subseteq L$.

(4) If $IAI\subseteq I$, then $I$ is an inner ideal of $L$. \end{lem}
\begin{proof}
(1) Recall that $I$ is an inner ideal of $L$ if and only if $[i,[j,x]]\in I$
for all $i,j\in I$ and all $x\in L$. It remains to note that $[i,[j,x]]=ijx-ixj-jxi+xji=-ixj-jxi$. 

(2) This follows from (1) because $ixj+jxi=(i+j)x(i+j)-ixi-jxj$. 

(3) Indeed, $iaj=i(aj)-(aj)i=[i,aj]\subseteq[A,A]=L$ for all $i,j\in I$
and all $a\in A$. 

(4) This follows from (2). \end{proof}
\begin{lem}
\label{reglemma*}Let $A$ be an associative algebra with involution
and let $K=\mathfrak{su}^{*}(A).$ Let $I$ be a subspace of $K$
such that $I^{2}=0.$ Then the following hold.

(1) $\mathfrak{u}^{*}(IAI)\subseteq K$.

(2) $\mathfrak{u}^{*}(IAI)=IAI\cap K$.

(3) If $\mathfrak{u}^{*}(IAI)\subseteq I$, then $I$ is an inner
ideal of $K$. \end{lem}
\begin{proof}
(1) Note that $IAI$ is $*$-invariant, so $\mathfrak{u}^{*}(IAI)=\{q-q^{*}\mid q\in IAI\}$.
It remains to note that 
\[
iaj-(iaj)^{*}=iaj-ja^{*}i=i(aj+ja^{*})-(aj+ja^{*})i=[i,aj-(aj)^{*}]\in[\mathfrak{u}^{*}(A),\mathfrak{u}^{*}(A)]=K
\]
 for all $i,j\in I$ and all $a\in A$. 

(2) This is obvious.

(3) By Lemma \ref{reglemma}(1), it is enough to check that $ixj+jxi\in I$
for all $i,j\in I$ and all $x\in K$. One has 
\[
ixj+jxi=ixj-(ixj)^{*}\in\mathfrak{u}^{*}(IAI)\subseteq I
\]
as required. 
\end{proof}
We will show (see Theorem \ref{general}) that for every inner ideal
$I$ of an infinite dimensional simple locally finite Lie algebra
$L$ one has $I^{2}=0$ (the only exception being the finitary orthogonal
algebras). Thus Lemmas \ref{reglemma} and \ref{reglemma*} justify
the following definition.
\begin{defn}
\label{regular} (1) Let $A$ be an associative algebra and let $L=[A,A]$.
Let $I$ be a subspace of $L$ such that $I^{2}=0.$ We say that $I$
is a \emph{regular }inner ideal of $L$ (with respect to $A$) if
and only if $IAI\subseteq I$. 

(2) Let $A$ be an associative algebra with involution and let $K=\mathfrak{su}^{*}(A)$.
Let $I$ be a subspace of $K$ such that $I^{2}=0.$ We say that $I$
is a \emph{$*$-regular} (or, simply, \emph{regular}) inner ideal
of $K$ (with respect to $A$) if and only if $\mathfrak{u}^{*}(IAI)\subseteq I$. \end{defn}
\begin{rem}
Note that regular inner ideals are always abelian (since $[I,I]\subseteq I^{2}=0$),
so they are proper inner ideals of $L$ (if $L$ is not abelian). 
\end{rem}
Recall that any associative algebra $B$ can be considered as a Lie algebra 
(denoted $B^{(-)}$) with respect to the new product 
$[a,b]=ab-ba$.  We will use the following well-known facts. 
\begin{lem}
\label{B+B*} Let $A$ be an associative algebra. 

(1) If $A$ is involution simple then $A$ is either simple or $A=B\oplus B^{*}$
where $B$ is a simple ideal. 

(2) Assume $A=B\oplus B^{*}$. Then $\mathfrak{u}^{*}(A)=\{(b,-b^{*})\mid b\in B\}$.
Let $\varphi$ be the projection of $A$ on $B$. Then the restriction
of $\varphi$ to $\mathfrak{u}^{*}(A)$ is an isomorphism of the Lie
algebras $\mathfrak{u}^{*}(A)$ and $B^{(-)}$. Moreover, if $C$
is a $*$-invariant subalgebra of $A$ then $\varphi(\mathfrak{u}^{*}(C))=\varphi(C)^{(-)}$. \end{lem}
\begin{proof}
(1) Suppose $A$ is not simple. So $A$ has a proper non-zero ideal
$B$. Then $B+B^{*}$ and $B\cap B^{*}$ are $*$-invariant ideals
of $A$. Since $A$ is involution simple, $B+B^{*}=A$ and $B\cap B^{*}=0$.
So $A=B\oplus B^{*}$ and $B$ is a simple ideal. 

(2) This is obvious. \end{proof}
\begin{lem}
\label{B+B* inner}Let $A=B\oplus B^{*}$ and let $\varphi:\mathfrak{su}^{*}(A)\rightarrow[B,B]$
be the isomorphism in Lemma \ref{B+B*}. Then $I$ is a regular
inner ideal of $\mathfrak{su}^{*}(A)$ if and only if $\varphi(I)$
is a regular inner ideal of $[B,B]$.\end{lem}
\begin{proof}
We need to show that $\mathfrak{u}^{*}(IAI)\subseteq I$ if and only
if $\varphi(I)B\varphi(I)\subseteq\varphi(I)$. Since both $\mathfrak{u}^{*}(IAI)$
and $I$ are subspaces of $\mathfrak{u}^{*}(A)$, the first inclusion
is equivalent to $\varphi(\mathfrak{u}^{*}(IAI))\subseteq\varphi(I)$.
Note that $\varphi(\mathfrak{u}^{*}(IAI))=\varphi(IAI)=\varphi(I)B\varphi(I)$,
so this can be rewritten as $\varphi(I)B\varphi(I)\subseteq\varphi(I)$,
as required. \end{proof}
\begin{lem}
\label{RLgeneral}Let $A$ be a simple associative ring and let $\mathcal{L}$
(resp. $\mathcal{R}$) be a left (resp. right) non-zero ideal of $A$.
Then the following holds.

(1) $\mathcal{L}A=A$, $A\mathcal{R}=A$, and $\mathcal{L}A\mathcal{R}=A$.

(2) $\mathcal{R}\mathcal{L}\subseteq\mathcal{R}\cap\mathcal{L}$.

(3) If $\mathcal{LR}=0$ then $\mathcal{R}\mathcal{L}\subseteq\mathcal{R}\cap\mathcal{L}\cap[A,A]$. 

(4) $\mathcal{R}\mathcal{L}$ and $\mathcal{R}\cap\mathcal{L}$ are
non-zero.\end{lem}
\begin{proof}
(1) Assume $\mathcal{L}A\ne A$. Since $\mathcal{L}A$ is a two-sided
ideal of $A$ and $A$ is simple, $\mathcal{L}A=0$. This implies
that $\mathcal{L}$ is a two-sided ideal of $A$. Since $\mathcal{L}$
is non-zero, $\mathcal{L}=A,$ so $\mathcal{L}A=A^{2}=A\ne0,$ which
is a contradiction. The proof for $\mathcal{R}$ is similar. Now $\mathcal{L}A\mathcal{R}=(\mathcal{L}A)(A\mathcal{R})=A^{2}=A.$

(2) It is enough to note that $\mathcal{R}\mathcal{L}\subseteq\mathcal{R}$
and $\mathcal{R}\mathcal{L}\subseteq\mathcal{L}$;

(3) This is obvious. 

(4) Assume $\mathcal{R}\mathcal{L}=0$. Then $A=A^{2}=(A\mathcal{R})(\mathcal{L}A)=A(\mathcal{R}\mathcal{L})A=0,$
which is a contradiction. 
\end{proof}
Let $A$ be an associative ring. An element $x\in A$ is called \emph{von
Neumann regular} if there is $y\in A$ such that $xyx=x$. The ring
$A$ is called \emph{von Neumann regular} if every element of $A$ is von
Neumann regular. We are grateful to Miguel G\'{o}mez Lozano for the
following observation. 
\begin{prop}
\label{vN} Let $A$ be an associative ring. 

(1) $\mathcal{R}\mathcal{L}=\mathcal{R}\cap\mathcal{L}$ for all left
and right ideals $\mathcal{L}$ and $\mathcal{R}$, respectively,
in $A$ if and only if $A$ is von Neumann regular. 

(2) $\mathcal{R}\mathcal{L}=\mathcal{R}\cap\mathcal{L}$ for all left
and right ideals $\mathcal{L}$ and $\mathcal{R}$, respectively,
in $A$ such that $\mathcal{LR}=0$ if and only if every $x$ in $A$
with $x^{2}=0$ is von Neumann regular. \end{prop}
\begin{proof}
(1) Suppose $\mathcal{R}\mathcal{L}=\mathcal{R}\cap\mathcal{L}$ for
all left and right ideals $\mathcal{L}$ and $\mathcal{R}$, respectively.
Let $x\in A$. Consider the ideals $\mathcal{R}=xA+\mathbb{Z}x$ and
$\mathcal{L}=Ax+\mathbb{Z}x$. Note that $x\in\mathcal{R}\cap\mathcal{L}=\mathcal{R}\mathcal{L}=xAx+\mathbb{Z}x^{2}$.
Hence $xy'x=x$ for some $y'\in A'$ where the ring $A'=A+\mathbb{Z}{\bf 1}$
is obtained from $A$ by adding the identity element ${\bf 1}$. Since $A$
is an ideal of $A'$, one has $xyx=x$ for $y=y'xy'\in A$. Therefore
$A$ is von Neumann regular. 

Assume now $A$ is von Neumann regular. Let $\mathcal{L}$ and $\mathcal{R}$
be a left and right ideal of $A$ respectively. Clearly $\mathcal{RL}\subseteq\mathcal{R}\cap\mathcal{L}$.
Let $x\in\mathcal{R}\cap\mathcal{L}$. Then there exists $y\in A$
such that $x=xyx=x(yx)\in\mathcal{RL}$. So $\mathcal{R}\mathcal{L}=\mathcal{R}\cap\mathcal{L}$. 

(2) Suppose $\mathcal{R}\mathcal{L}=\mathcal{R}\cap\mathcal{L}$ for
all left and right ideals $\mathcal{L}$ and $\mathcal{R}$, respectively,
such that $\mathcal{LR}=0$. Let $x\in A$ with $x^{2}=0$. Then $\mathcal{R=}xA+\mathbb{Z}x$
and $\mathcal{L}=Ax+\mathbb{Z}x$ are a left and right ideal of $A$
with $\mathcal{LR}=0$. The rest of the argument follows as in (1).

Assume now every $x\in A$ for which $x^{2}=0$ is von Neumann regular.
Let $\mathcal{L}$ and $\mathcal{R}$ be a left and right ideal of
A respectively with $\mathcal{LR}=0$. Clearly $\mathcal{RL}\subseteq\mathcal{R}\cap\mathcal{L}$.
Let $x\in\mathcal{R\cap L}$. Note that $x^{2}\in\mathcal{LR}=0$
so $x$ by assumption is von Neumann regular. So there exists $y\in A$
such that $x=xyx=x(yx)\in\mathcal{RL}.$ Therefore $\mathcal{R}\mathcal{L}=\mathcal{R}\cap\mathcal{L}$. 
\end{proof}
Now we are in position to describe regular inner ideals.
\begin{prop}
\label{RLinner} Let $A$ be an associative algebra and let $L=[A,A]$.
Let $I$ be a subspace of $L$. Then $I$ is a regular inner ideal
of $L$ if and only if 
\begin{equation}
\mathcal{R}\mathcal{L}\subseteq I\subseteq\mathcal{R}\cap\mathcal{L}\cap L\label{eq:RL}
\end{equation}
where $\mathcal{L}$ (resp. $\mathcal{R}$) is a left (resp. right)
ideal of $A$ such that $\mathcal{LR}=0$. In particular, if $A$
is von Neumann regular then every regular inner ideal of $L$ is of
the form 
$I=\mathcal{R}\mathcal{L}=\mathcal{R}\cap\mathcal{L}$
for some left ideal $\mathcal{L}$ and right ideal $\mathcal{R}$.
\end{prop}
\begin{proof}
Assume first that $I$ is a regular inner ideal of $L$. Then $I^{2}=0$
and $IAI\subseteq I$. Put $\mathcal{L}=AI+I$ and $\mathcal{R}=IA+I$.
Then $\mathcal{L}$ (resp. $\mathcal{R}$) is a left (resp. right)
ideal of $A$ with $\mathcal{LR}=0$ and $\mathcal{R}\mathcal{L}\subseteq IAI+I=I\subseteq\mathcal{R}\cap\mathcal{L}\cap L$,
as required. 

Now assume that $\mathcal{R}\mathcal{L}\subseteq I\subseteq\mathcal{R}\cap\mathcal{L}\cap L$.
Then $I^{2}\subseteq\mathcal{LR}=0$ and $IAI\subseteq\mathcal{R}A\mathcal{L}\subseteq\mathcal{R}\mathcal{L}\subseteq I$,
so $I$ is a regular inner ideal.
\end{proof}
If $A$ is simple then one can show that the ideals $\mathcal{L}$
and $\mathcal{R}$ are defined by $I$ almost uniquely. More exactly
we have the following. 
\begin{lem}
\label{ALeqAI} Let $A$ be a simple associative algebra and let $L=[A,A]$.
If $I$ is a regular inner ideal of $L$ and a pair of ideals $(\mathcal{L},\mathcal{R})$
satisfies (\ref{eq:RL}) then $A\mathcal{L}=AI$ and $\mathcal{R}A=IA$.\end{lem}
\begin{proof}
Assume the pair of ideals $(\mathcal{L},\mathcal{R})$ satisfies (\ref{eq:RL}).
Then $I\subseteq\mathcal{L}$, so $AI\subseteq A\mathcal{L}$. On
the other hand, by Lemma \ref{RLgeneral}(1), $A\mathcal{L}=(\mathcal{L}A\mathcal{R})\mathcal{L}=\mathcal{L}A(\mathcal{R}\mathcal{L})\subseteq AI$,
so $A\mathcal{L}=AI$. Similarly, $\mathcal{R}A=IA$.
\end{proof}
The next proposition describes regular inner ideals in the case of
algebras with involution.
\begin{prop}
\label{RLinner*} Let $A$ be an associative algebra with involution
and let $K=\mathfrak{su}^{*}(A)$. Let $I$ be a subspace of $K$.
Then $I$ is a regular inner ideal of $K$ if and only if 
\begin{equation}
\mathfrak{u}^{*}(\mathcal{L}^{*}\mathcal{L})\subseteq I\subseteq\mathcal{L}^{*}\cap\mathcal{L}\cap K\label{eq:RL*}
\end{equation}
 where $\mathcal{L}$ is a left ideal of $A$ such that $\mathcal{LL}^{*}=0$.
In particular, if $A$ is von Neumann regular then every regular inner
ideal of $L$ is of the form $I=\mathfrak{u}^{*}(\mathcal{L}^{*}\mathcal{L})=
\mathfrak{u}^{*}(\mathcal{L}^{*}\cap\mathcal{L})$
for some left ideal $\mathcal{L}$. 
\end{prop}
\begin{proof}
Assume first that $I$ is a regular inner ideal of $K$. Then $I^{2}=0$
and $\mathfrak{u}^{*}(IAI)\subseteq I$. Put $\mathcal{L}=AI+I$.
Then $\mathcal{L}$ is a left ideal of $A$, $\mathcal{L}^{*}=IA+I$,
$\mathcal{LL}^{*}=0$ and $\mathfrak{u}^{*}(\mathcal{L}^{*}\mathcal{L})\subseteq\mathfrak{u}^{*}(IAI+I)\subseteq I\subseteq\mathcal{L}^{*}\cap\mathcal{L}\cap K$
, as required. 

Now assume that $\mathfrak{u}^{*}(\mathcal{L}^{*}\mathcal{L})\subseteq I\subseteq\mathcal{L}^{*}\cap\mathcal{L}\cap K$.
Then $I^{2}\subseteq\mathcal{LL}^{*}=0$ and $\mathfrak{u}^{*}(IAI)\subseteq\mathfrak{u}^{*}(\mathcal{L}^{*}A\mathcal{L})\subseteq\mathfrak{u}^{*}(\mathcal{L}^{*}\mathcal{L})\subseteq I$,
so $I$ is a regular inner ideal. \end{proof}
\begin{prop}
\label{semisimpleReg} Let $A$ be a finite dimensional semisimple
associative algebra and let $L=[A,A].$ Then every proper inner ideal $I$
of $L$ is regular. More exactly, $I=\mathcal{\mathcal{R}L}$ ($=\mathcal{L}\cap\mathcal{R}$)
where $\mathcal{L}$ is a left ideal of $A$ and $\mathcal{R}$ is
a right ideal of $A$ with $\mathcal{LR}=0$.\end{prop}
\begin{proof}
Suppose $I$ is an inner ideal of $L$. Note that $L$ is semisimple.
Therefore by Propositions \ref{InnerDecomposition} and \ref{ClassInnerFD}(1),
$I=eAf$ for a pair of idempotents $e$ and $f$ of $A$ such that
$fe=0$. Define $\mathcal{L}=Af$ and $\mathcal{R}=eA$. Then $\mathcal{R}\mathcal{L}=eAAf=eAf=I$,
as required. \end{proof}
\begin{thm}
\label{general} (1) Let $A$ be a simple locally finite associative
algebra and let $(A_{\alpha})_{\alpha\in\Gamma}$ be a perfect conical
local system for $A$ of rank $>4$. Let $L=[A,A]$ and let $I$ be
a proper inner ideal of $L$. Put $L_{\alpha}=[A_{\alpha},A_{\alpha}]$
and $I_{\alpha}=I\cap L_{\alpha}$. Let $\overline{I_{\alpha}}$ be
the image of $I_{\alpha}$ in $\overline{L_{\alpha}}=L_{\alpha}/\Rad L_{\alpha}$.
Then $I^{2}=0$ and for every $\alpha\in\Gamma$, $\overline{I_{\alpha}}$
is a regular inner ideal of $\overline{L_{\alpha}}$. 

(2) Let $A$ be an involution simple locally finite associative algebra
and let $(A_{\alpha})_{\alpha\in\Gamma}$ be a perfect conical $*$-invariant
local system for $A$ of rank $>36$. Let $L=\mathfrak{su}^{*}(A)$
and let $I$ be an inner ideal of $L$. Put $L_{\alpha}=\mathfrak{su}^{*}(A_{\alpha})$
and $I_{\alpha}=I\cap L_{\alpha}$. Let $\overline{I_{\alpha}}$ be
the image of $I_{\alpha}$ in $\overline{L_{\alpha}}=L_{\alpha}/\Rad L_{\alpha}$.
If $A$ is not finitary with orthogonal involution (i.e. $L$ is not
finitary orthogonal) then $I^{2}=0$ and there is $\alpha_{0}\in\Gamma$
such that for every $\alpha\ge\alpha_{0}$, $\overline{I_{\alpha}}$
is a regular inner ideal of $\overline{L_{\alpha}}$. \end{thm}
\begin{proof}
(1) By \cite[Theorem 6.3(1)]{bib:BZ1} (see also \cite[Theorem 2.12(1)]{bib:BBZ}
and its proof), $(L_{\alpha})_{\alpha\in\Gamma}$ is a perfect conical
local system for $L$. By Proposition \ref{Inner0} $I_{\alpha}$
is an inner ideal of $L_{\alpha}$. By Lemma \ref{radical}, for every
$\alpha$ there exists $\beta$ such that $\Rad A_{\beta}\cap A_{\alpha}=0$.
Let $^{-}:A_{\beta}\rightarrow A_{\beta}/\Rad A_{\beta}$ be the canonical
surjection. Note that $\Rad L_{\beta}\subseteq\Rad A_{\beta}$, $\overline{L_{\beta}}=[\overline{A_{\beta}},\overline{A_{\beta}}]$,
and by Lemma \ref{Inner0}, $\overline{I_{\beta}}$ is an inner ideal
of $\overline{L_{\beta}}$. Moreover $\overline{I_{\beta}}$ is regular
by Proposition \ref{semisimpleReg} and $\overline{I_{\beta}}^{2}=0$.
Since $A_{\alpha}\cap\Rad A_{\beta}=0$, $\overline{A}_{\alpha}\cong A_{\alpha}$,
so $A_{\alpha}$ can be considered as a subalgebra of $\overline{A_{\beta}}$
and $I_{\alpha}\subseteq\overline{I_{\beta}}$. Therefore $I_{\alpha}^{2}\subseteq\overline{I_{\beta}}^{2}=0$,
so $I_{\alpha}^{2}=0$. Since $I=\dlim I_{\alpha}$, we conclude that
$I^{2}=0$. This implies that $\overline{I_{\alpha}}$ is a proper
inner ideal of $\overline{L_{\alpha}}$ for every $\alpha\in\Gamma$,
so $\overline{I_{\alpha}}$ is regular by Proposition \ref{semisimpleReg}. 

(2) By \cite[Theorem 6.3]{bib:BZ2} (see also \cite[Theorem 2.12(2)]{bib:BBZ}
and its proof), $(L_{\alpha})_{\alpha\in\Gamma}$ is a perfect conical
local system for $L$. By Proposition \ref{Inner0}, $I_{\alpha}$
is an inner ideal of $L_{\alpha}$. Assume first that $A$ is not
simple. Then by Lemma \ref{B+B*} $A=B\oplus B^{*}$ where $B$ is
a simple ideal of $A$. Moreover, if $\varphi$ is the projection
of $A$ on $B$ then $\varphi$ is an isomorphism of the Lie algebras
$\mathfrak{su}^{*}(A)$ and $[B,B]$ and the result follows from part
(1) of the theorem. Thus, we can suppose that $A$ is simple. 

Assume now that $L$ is finitary. Since $A$ is simple and the involution
is not orthogonal, it must be symplectic. Therefore there is a local
system $(S_{\delta})_{\delta\in\Delta}$ of naturally embedded finite
dimensional symplectic subalgebras of $L$. Fix any $\delta\in\Delta$
and $\alpha_{0}\in\Gamma$ such that $S_{\delta}$ is of rank $>10$
and $L_{1}\subseteq S_{\delta}\subseteq L_{\alpha_{0}}$. We claim
that $\overline{L_{\alpha}}=L_{\alpha}/\Rad L_{\alpha}$ is symplectic
for all $\alpha\ge\alpha_{0}$. Indeed, consider any Levi subalgebra
$Q$ of $L_{\alpha}$ which contains $S_{\delta}$ and fix $\delta'$
such that $Q\subseteq S_{\delta'}$. We have a chain of embeddings
\[
L_{1}\subseteq S_{\delta}\subseteq Q\subseteq S_{\delta'}.
\]
Since the embedding $S_{\delta}\subseteq S_{\delta'}$ is diagonal,
by Lemma \ref{BBZ2.5}, both embeddings $S_{\delta}\subseteq Q$ and
$Q\subseteq S_{\delta'}$ are diagonal. Moreover, since $S_{\delta}\subseteq S_{\delta'}$
is natural, $Q$ must be simple and both embeddings $S_{\delta}\subseteq Q$
and $Q\subseteq S_{\delta'}$ must be natural. This implies that $Q$
is symplectic (see for example \cite[Proposition 2.3]{bib:BZh}),
so $\overline{L_{\alpha}}\cong Q$ is symplectic. Therefore $\overline{I_{\alpha}}$
is a regular inner ideal of $\overline{L_{\alpha}}$ and $\overline{I_{\alpha}}^{2}=0$
for all $\alpha\ge\alpha_{0}$. As in the proof of part (1), fix any
$\beta$ such that $\Rad A_{\beta}\cap A_{\alpha}=0$. Then $I_{\alpha}^{2}\subseteq\overline{I_{\beta}}^{2}=0$,
so $I^{2}=0$. 

Suppose now that $L$ is not finitary. First we are going to show
that $I^{2}=0$. Assume $I^{2}\ne0$. Fix any $x,y\in I$ such that
$xy\ne0$. Since $I^{2}=\dlim I_{\alpha}^{2}$, there is $\beta\in\Gamma$
such that $x,y\in I_{\gamma}$ and $xy\notin\Rad A_{\gamma}$ for
all $\gamma\ge\beta$. Let $M$ be a $*$-invariant maximal ideal
of $A_{\gamma}$ with $xy\notin M$. Note that $Q=A_{\gamma}/M$ is
involution simple and $K=\mathfrak{su}^{*}(Q)$ is isomorphic to one
of the simple components of $L_{\gamma}/\Rad L_{\gamma}$. Let $V$
be the corresponding natural module for $K$ and $L_{\gamma}$ and
let $J$ be the image of $I_{\gamma}$ in $K$. Then $J$ is an inner
ideal of $K$. Since $xy\notin M$, $J^{2}$ is nonzero in $Q$. Therefore
$J$ is as in Theorem \ref{ClassInnerFD}(3)(ii), i.e. spanned by
the matrix units $e_{1j}-e_{j2}$, $j\ge3$. In particular, $x$ (and
$y$) is of rank $2$ on $V$. Thus $x$ acts as zero or a rank 2
linear transformation on every natural $L_{\gamma}$-module for all
$\gamma\ge\beta$. Therefore by Theorem \ref{finitaryCriterion},
$L$ is finitary, which contradicts the assumption. 

Fix any non-zero $x\in I$ and any $\beta\in\Gamma$ such that $x\in I_{\gamma}$
and $x\notin\Rad L_{\gamma}$ for all $\gamma\ge\beta$. One can also
assume that $x$ is of rank greater than 2 on some natural $L_{\beta}$-module
$V$ (otherwise $L$ is finitary by Theorem \ref{finitaryCriterion}).
Let $Q$ be the corresponding simple component of a Levi subalgebra
of $L_{\beta}$ (so $V$ is a natural $Q$-module). By Corollary \ref{SimpleCrit}
there exists $\alpha_{0}\in\Gamma$ such that for all $\alpha\ge\alpha_{0}$
the restriction of every natural $L_{\alpha}$-module $W$ to $Q$
has a non-trivial composition factor. Since the embedding $L_{\beta}\subseteq L_{\alpha}$
is diagonal, this implies that the restriction of $W$ to $L_{\beta}$
contains $V$ or $V^{*}$ as a composition factor, so rank of $x$
on $W$ is greater than 2. Let $M$ be the annihilator of $W$ in
$L_{\alpha}$. Then $M$ is a maximal ideal. Note that the image $J$
of $I_{\alpha}$ in $S=L_{\alpha}/M$ is a regular inner ideal of
$S$ because it contains the non-zero image of $x$ and rank of $x$
is greater than 2 on the natural $S$-module $W$. Since the intersection
of all maximal ideals of $L_{\alpha}$ is the radical of $L_{\alpha}$
this implies that $\overline{I_{\alpha}}$ is a regular inner ideal
of $\overline{L_{\alpha}}$. 
\end{proof}

Recall that every simple diagonal locally finite Lie algebra can be
represented as $\mathfrak{su}^{*}(A)$ where $A$ is an involution
simple locally finite associative algebra ($A$ is actually unique
and called the $\mathfrak{P}^{*}$-enveloping algebra of $L$, see
Introduction). Moreover, if $L$ is plain then $L=[A,A]$ where $A$
is a $\mathfrak{P}$-enveloping algebra of $L$. Thus, the theorem
above describes inner ideals of simple diagonal locally finite Lie algebras. 
In particular it shows that $I^2=0$ (in $A$) 
for any proper inner ideal $I$ in such Lie algebra $L$. 
Following \cite{FL}, we say that a subspace $I$ of an associative algebra $A$
is a {\em Jordan-Lie} inner ideal of $A$ if $I^2=0$ and 
$I$ is an inner ideal of the Lie algebra $A^{(-)}$ 
(this actually implies that $I$ is an inner ideal of the Jordan algebra $A^{+}$, 
which explains the name). We are grateful to Antonio Fern\'{a}ndez L\'{o}pez 
for the following observation. 

\begin{cor}
Let $L$ be a simple plain Lie algebra and let $A$ be its simple
associative $\mathfrak{P}$-envelope. Then every proper inner ideal $I$ 
of $L$ is a Jordan-Lie inner ideal of $A$.
\end{cor}

\begin{proof}
Recall that $L=[A,A]$. By Theorem \ref{general}(1), one has $I^2=0$
and, by Lemma \ref{reglemma}(2), $xLx\subseteq I$ for all $x\in I$. 
We need to show that $I$ is an inner ideal of $A^{(-)}$,
i.e. $xAx\subseteq I$ for all $x\in I$. 
Since $A$ is simple, it is the linear span of the elements of the form
$axb$, $a,b\in A$. As $I^2=0$, one has
$x(axb)x=x[a,xb]x\in xLx\subseteq I$, as required. 
\end{proof}

We say that an associative algebra with involution is \emph{$*$-locally
semisimple} if it has a local system of $*$-invariant semisimple
finite dimensional subalgebras. 
Note that there are examples of simple locally finite Lie algebras which are 
(1) diagonal but not locally semisimple (see \cite{bib:BStr}) and 
(2) diagonal and  locally semisimple
but not finitary (see \cite{bib:BZh}).

\begin{prop}
\label{LocSemiEnv}(1) Let $L$ be a simple diagonal Lie algebra and
let $A$ be its involution simple associative $\mathfrak{P}^{*}$-envelope.
Then $L$ is locally semisimple if and only if $A$ is $*$-locally
semisimple. 

(2) Let $L$ be a simple plain Lie algebra and let $A$ be its simple
associative $\mathfrak{P}$-envelope. Then $L$ is locally semisimple
if and only if $A$ is locally semisimple.\end{prop}
\begin{proof}
We will only prove the first part. The proof of the second statement
is similar. Assume first that $A$ is $*$-locally semisimple. Then
$A$ has a local system $(A_{\alpha})_{\alpha\in\Gamma}$ such that
all $A_{\alpha}$ are $*$-invariant semisimple finite dimensional
algebras. Let $L_{\alpha}=\mathfrak{su}^{*}(A_{\alpha})$. Then $L_{\alpha}$
is a semisimple finite dimensional Lie algebra for each $\alpha$
(see \cite[Lemma 2.3]{bib:BZ2} for example). Therefore $(L_{\alpha})_{\alpha\in\Gamma}$
is a semisimple local system for $L$ and $L$ is locally semisimple. 

Assume now that $L$ is locally semisimple. By Proposition \ref{diag_sys},
$L$ has a diagonal semisimple conical local system $(L_{\alpha})_{\alpha\in\Gamma}$
of rank >10. It follows from the construction of $A$ as a quotient
of the universal enveloping algebra $U(L)$ by the annihilator of
a diagonal inductive system for $L$ (see proof of \cite[Theorem 1.3]{bib:BBZ})
that $A$ is $*$-locally semisimple. \end{proof}
\begin{cor}
\label{LocSemi} (1) Let $L$ be a simple plain Lie algebra and let
$A$ be its simple associative $\mathfrak{P}$-envelope, so $L=[A,A]$.
Assume that $L$ is locally semisimple. Then the following hold.

(i) $A$ is locally semisimple and von Neumann regular. 

(ii) Every proper inner ideal $I$ of $L$ is regular, i.e. $I=\mathcal{\mathcal{R}L}$
($=\mathcal{L}\cap\mathcal{R}$) where $\mathcal{L}$ is a left ideal
of $A$ and $\mathcal{R}$ is a right ideal of $A$ with $\mathcal{LR}=0$. 

(iii) A subspace $I$ of $L$ is a proper inner ideal of $L$ if and only
if $I=\dlim e_{\alpha}Af_{\alpha}$ where $\{e_{\alpha},f_{\alpha}\mid\alpha\in B\}$
is a directed system of idempotents in $A$ such that $f_{\alpha}e_{\alpha}=0$,
$e_{\beta}e_{\alpha}=e_{\alpha}$ and $f_{\alpha}f_{\beta}=f_{\alpha}$
for all $\alpha,\beta$ with $\alpha\le\beta$. 

(2) Let $L$ be a simple diagonal Lie algebra and let $A$ be its
involution simple associative $\mathfrak{P}^{*}$-envelope, so $L=\mathfrak{su}^{*}(A)$.
Assume that $L$ is locally semisimple. Then the following hold.

(i) $A$ is $*$-locally semisimple and von Neumann regular. 

(ii) If $L$ is not finitary orthogonal then every proper inner ideal
$I$ of $L$ is regular, i.e. $I=\mathfrak{u}^{*}(\mathcal{L}^{*}\mathcal{L})$
($=\mathfrak{u}^{*}(\mathcal{L}^{*}\cap\mathcal{L})$) where $\mathcal{L}$
is a left ideal of $A$ with $\mathcal{LL}^{*}=0$. 

(iii) If $L$ is not finitary orthogonal then a subspace $I$ of $L$
is a proper inner ideal of $L$ if and only if $I=\dlim\mathfrak{u}^{*}(e_{\alpha}Ae_{\alpha}^{*})$
where $\{e_{\alpha}\mid\alpha\in B\}$ is a directed system of idempotents
in $A$ such that $e_{\alpha}^{*}e_{\alpha}=0$ and $e_{\beta}e_{\alpha}=e_{\alpha}$
for all $\alpha,\beta$ with $\alpha\le\beta$. \end{cor}
\begin{proof}
We will prove part (2) only. Proof of part (1) is similar. 

(i) By Proposition \ref{LocSemiEnv}, $A$ is $*$-locally semisimple,
so von Neumann regular. 

(ii) Let $(A_{\alpha})_{\alpha\in\Gamma}$ be a $*$-invariant semisimple
local system for $A$. By \cite[2.9-2.11]{bib:BBZ} we can assume
that this local system is conical of rank $>36$. Then the Lie algebras
$L_{\alpha}=\mathfrak{su}^{*}(A_{\alpha})$ are semisimple for all
$\alpha$ and $(L_{\alpha})_{\alpha\in\Gamma}$ is a conical local
system of $L$. Let $I$ be any inner ideal of $L$ and let $I_{\alpha}=I\cap L_{\alpha}$.
By Theorem \ref{general}(2), there is $\alpha_{0}\in\Gamma$ such
that $I_{\alpha}=\overline{I_{\alpha}}$ is a regular inner ideal
of $L_{\alpha}$ for all $\alpha\ge\alpha_{0}$. We need to show that
$I$ is regular, i.e. $\mathfrak{u}^{*}(IAI)\subseteq I$. Consider
any element $x\in\mathfrak{u}^{*}(IAI)$. Then there exists $\alpha\ge\alpha_{0}$
such that $x\in\mathfrak{u}^{*}(I_{\alpha}A_{\alpha}I_{\alpha})$.
Since $I_{\alpha}$ is a regular inner ideal, $x\in I_{\alpha}\subseteq I$.
Hence $I$ is a regular inner ideal. By Proposition \ref{RLinner*}
all regular inner ideals of $L$ are of the form $I=\mathfrak{u}^{*}(\mathcal{L}^{*}\mathcal{L})$
($=\mathfrak{u}^{*}(\mathcal{L}^{*}\cap\mathcal{L})$ ) where $\mathcal{L}$
is a left ideal of $A$ with $\mathcal{LL}^{*}=0$. 

(iii) Assume first that $I$ is an inner ideal of $L$ and let $(A_{\alpha})_{\alpha\in\Gamma}$
be a $*$-invariant semisimple local system for $A$. Then $I$ is
regular by part (ii), so $I=\mathfrak{u}^{*}(\mathcal{L}^{*}\mathcal{L})$
where $\mathcal{L}$ is a left ideal of $A$ with $\mathcal{LL}^{*}=0$.
Let $\mathcal{L}_{\alpha}=\mathcal{L}\cap A_{\alpha}$. Since every
one-sided ideal of a finite dimensional semisimple algebra is generated
by an idempotent, $\mathcal{L}_{\alpha}=A_{\alpha}e_{\alpha}^{*}$
and $\mathcal{L}_{\alpha}^{*}=e_{\alpha}A_{\alpha}$ for some idempotent
$e_{\alpha}$ of $A_{\alpha}$. We claim that the system $\{e_{\alpha}\mid\alpha\in\Gamma\}$
satisfies the required conditions. Let $\beta\ge\alpha$. Recall that
$A_{\alpha}$ is semisimple so it contains the identity element ${\bf {1}}$,
so $e_{\alpha}=e_{\alpha}{\bf {1}}\in e_{\alpha}A_{\alpha}\subseteq e_{\beta}A_{\beta}$. 
Since $e_{\beta}x=x$ for all $x\in\mathcal{L}_{\beta}^{*}=e_{\beta}A_{\beta}$
we have that $e_{\beta}e_{\alpha}=e_{\alpha}$. Also we have 
\[
e_{\alpha}^{*}e_{\alpha}\in A_{\alpha}e_{\alpha}^{*}e_{\alpha}A_{\alpha}=\mathcal{L}_{\alpha}\mathcal{L}_{\alpha}^{*}\subseteq\mathcal{LL}^{*}=0
\]
so $e_{\alpha}^{*}e_{\alpha}=0$. Note that $e_{\alpha}A_{\beta}=e_{\beta}e_{\alpha}A_{\beta}\subseteq e_{\beta}A_{\beta}$
for all $\beta\ge\alpha$, so $e_{\alpha}A_{\beta}e_{\alpha}^{*}\subseteq e_{\beta}A_{\beta}e_{\beta}^{*}$.
Therefore 
\[
I=\mathfrak{u}^{*}(\mathcal{L}^{*}\mathcal{L})=\dlim\mathfrak{u}^{*}(\mathcal{L}_{\alpha}^{*}\mathcal{L}_{\alpha})=\dlim\mathfrak{u}^{*}(e_{\alpha}A_{\alpha}e_{\alpha}^{*})=\dlim\mathfrak{u}^{*}(e_{\alpha}Ae_{\alpha}^{*}),
\]
as required.

Assume now that $\{e_{\alpha}\mid\alpha\in B\}$ is a directed system
of idempotents in $A$ such that $e_{\alpha}^{*}e_{\alpha}=0$ and
$e_{\beta}e_{\alpha}=e_{\alpha}$ for all $\alpha,\beta$ with $\alpha\le\beta$.
Then $e_{\alpha}A$ is a right ideal of $A$ and $e_{\alpha}A=e_{\beta}e_{\alpha}A\subseteq e_{\beta}A$
for all $\beta\ge\alpha$. Therefore the one-sided ideals $\mathcal{L}=\dlim Ae_{\alpha}^{*}$
and $\mathcal{L}^{*}=\dlim e_{\alpha}A$ are well-defined. Note that
$\mathcal{LL}^{*}=0$, so $I=\mathfrak{u}^{*}(\mathcal{L}^{*}\mathcal{L})$
is a regular inner ideal of $L$ by Proposition \ref{RLinner*}. 
\end{proof}

\section{Finitary Lie algebras}

\label{sec:finitary}

Recall that an algebra is called \emph{finitary} if it consists of
finite-rank linear transformations of a vector space. First we define
the classical finitary simple Lie algebras, see \cite{bib:Finitary0}
and \cite{bib:Inner} for details.

A pair of dual vector spaces $(X,Y,g)$ consists of vector spaces
$X$ and $Y$ over $F$ and a non-degenerate bilinear form $g:X\times Y\rightarrow F$
. A linear transformation $a:X\rightarrow X$ is \emph{continuous}
(relative to $Y$) if there exists $a^{\#}:Y\rightarrow Y$, necessarily
unique, such that $g(ax,y)=g(x,a^{\#}y)$ for all $x\in X$, $y\in Y$.
Note that $Y$ can be identified with a \emph{total subspace} (i.e.
Ann$_{X}Y=0$) of the dual vector space $X^{*}$. In that case $a^{\#}\varphi=\varphi a$
for all $\varphi\in X^{*}$ and $a$ is continuous if and only if
$a^{\#}Y\subseteq Y$. 

Denote by $\mathcal{F}(X,Y)$ the algebra of all continuous (relative
to $Y$) finite rank linear transformations of $X$. Then $\mathcal{F}(X,Y)$
is a simple associative algebra with minimal left ideals. For $u\in X$,
$w\in Y$ we denote by $w^{*}u$ the linear transformation $w^{*}u(x)=g(x,w)u$,
$x\in X$, and for subspaces $U\subseteq X$ and $W\subseteq Y$ we
denote by $W^{*}U$ the set of all finite sums of $w_{i}^{*}u_{i}$,
$u_{i}\in U$, $w_{i}\in W$. Note that $(y_{2}^{*}x_{2})(y_{1}^{*}x_{1})=g(x_{1},y_{2})y_{1}^{*}x_{2}$,
for $x_{1},x_{2}\in X$, $y_{1},y_{2}\in Y$ and $\mathcal{F}(X,Y)=Y^{*}X$. 

The \emph{finitary special linear Lie algebra} $\mathfrak{fsl}(X,Y)$
is defined to be $[\mathcal{F}(X,Y),\mathcal{F}(X,Y)]$. 

Let $\Phi$ be a nondegenerate symmetric or skew-symmetric form on
$X$, $\Phi(y,x)=\epsilon\Phi(x,y)$, $\epsilon=\pm1$,
for $x,y\in X$. 
 Then $X$
becomes a self-dual vector space with respect to $\Phi$ and the algebra
$\mathcal{F}(X,X)$ of continuous linear transformations on $X$ has
an involution $a\mapsto a^{*}$ given by $\Phi(ax,y)=\Phi(x,a^{*}y)$,
for all $x,y\in X$. As before, we denote by $\mathfrak{u}^{*}(\mathcal{F}(X,X))=\{a\in\mathcal{F}(X,X)\mid a^{*}=-a\}$
the set of skew-symmetric elements of $\mathcal{F}(X,X)$ and by $\mathfrak{su}^{*}(\mathcal{F}(X,X))$
its derived subalgebra. For $x,y\in X$, define $[x,y]=x^{*}y-\epsilon y^{*}x\in\mathcal{F}(X,X)$.
One can check that $(x^{*}y)^{*}=\epsilon y^{*}x,$ so $[x,y]\in\mathfrak{u}^{*}(\mathcal{F}(X,X))$.
If $U,W$ are subspaces of $X$, then $[U,W]$ will denote the set
of all finite sums of $[u_{i},w_{i}]$, $u_{i}\in U$, $w_{i}\in W$.
Note that
\[
\mathfrak{u}^{*}(\mathcal{F}(X,X))=\{b-b^{*}\mid b\in\mathcal{F}(X,X)\}=\{x^{*}y-\epsilon y^{*}x\mid x,y\in X\}=[X,X].
\]
If $\Phi$ is a symmetric bilinear form, then $\mathfrak{u}^{*}(\mathcal{F}(X,X))=\mathfrak{su}^{*}(\mathcal{F}(X,X))$
is the \emph{finitary orthogonal algebra} $\mathfrak{fo}(X,\Phi)$. 

If $\Phi$ is a skew-symmetric bilinear form, then $\mathfrak{u}^{*}(\mathcal{F}(X,X))=\mathfrak{su}^{*}(\mathcal{F}(X,X))$
is the \emph{finitary symplectic algebra }$\mathfrak{fsp}(X,\Phi)$. 
\begin{thm}
\cite[Corollary 1.2]{bib:Finitary0} \label{FinitaryClass} Any infinite
dimensional finitary simple Lie algebra over $F$ is isomorphic to
one of the following:

(1) A finitary special linear Lie algebra $\mathfrak{fsl}(X,Y)$.

(2) A finitary symplectic algebra $\mathfrak{fsp}(X,\Phi)$. 

(3) A finitary orthogonal algebra $\mathfrak{fo}(X,\Phi)$. 
\end{thm}
In \cite{bib:Finitary} this result was extended to positive characteristic. 

The classification of inner ideals of finitary simple Lie algebras
was first obtained by Fern\'{a}ndez L\'{o}pez, Garc\'{i}a and G\'{o}mez
Lozano \cite{bib:Inner} (over arbitrary fields of characteristic
zero), with Benkart and Fern\'{a}ndez L\'{o}pez \cite{bib:Benk3}
settling later the missing case for orthogonal algebras. We provide
an alternative proof for the case of special linear and symplectic
algebras over an algebraically closed field of characteristic zero.
In the case of orthogonal algebras we can only describe regular inner
ideals. 
\begin{thm}
\cite[results 2.5, 3.6, 3.8]{bib:Inner}\cite[Theorem 6.6]{bib:Benk3}
\label{finitReg}Let $(X,Y,g)$ be a dual pair of infinite dimensional
vector spaces over $F$ and let $\Phi$ (resp. $\Psi$) be a nondegenerate
symmetric (resp. skew-symmetric) form on $X$. 

(1) A subspace $I$ is a proper inner ideal of $\mathfrak{fsl}(X,Y)$
if and only if $I=W^{*}U$ where the subspaces $U\subseteq X$ and
$W\subseteq Y$ are mutually orthogonal (i.e. $g(U,W)=0$) (or equivalently,
$I$ is a regular inner ideal). 

(2) A subspace $I$ is a proper inner ideal of $\mathfrak{fsp}(X,\Psi)$
if and only if $I=[U,U]$ for some totally isotropic subspace $U$
of $X$ (i.e. $\Psi(U,U)=0$) (or equivalently, $I$ is a regular
inner ideal). 

(3) A subspace $I$ is a proper inner ideal of $\mathfrak{fo}(X,\Phi)$
if and only if $I$ satisfies one of the following. 

(i) $I=[U,U]$ for some totally isotropic subspace $U\mbox{\ensuremath{\subseteq}}X$
(or equivalently, $I$ is a regular inner ideal). 

(ii) $I$ is a Type 1 point space of dimension greater than 1.

(iii) $I=[x,H^{\perp}]$ where $H$ is a hyperbolic plane in $X$
and $x$ is a non-zero isotropic vector in $H.$\end{thm}
\begin{proof}
Note that the simple infinite dimensional finitary Lie algebras are
locally semisimple Lie algebras, so we can use Theorem \ref{LocSemi}.
The associative algebras $\mathcal{F}(X,Y)$ are simple, with minimal
one-sided ideals, and, in particular, they are locally finite dimensional
(see for example \cite[Theorem 4.15.3]{bib:Jac}).

(1) Recall that $\mathfrak{fsl}(X,Y)=[\mathcal{F}(X,Y),\mathcal{F}(X,Y)]$.
In particular, $\mathfrak{fsl}(X,Y)$ is plain and $\mathcal{F}(X,Y)$
is its simple associative $\mathfrak{P}$-envelope. By Corollary \ref{LocSemi}(1)
a subspace $I$ of $\mathfrak{fsl}(X,Y)$ is a proper inner ideal
if and only if it is a regular inner ideal, i.e. there exists a left
ideal and a right ideal of $\mathcal{F}(X,Y)$, say $\mathcal{L}$
and $\mathcal{R}$, such that $I=\mathcal{R}\mathcal{L}=\mathcal{L}\cap\mathcal{R}$
and $\mathcal{LR}=0$. By \cite[Theorem 4.16.1]{bib:Jac}, every right
ideal of $\mathcal{F}(X,Y)$ is of the form $\mathcal{R}=Y^{*}U=\{a\in\mathcal{F}(X,Y)\mid aX\subseteq U\}$
for some subspace $U\subseteq X$ and every left ideal is of the form
$\mathcal{L}=W^{*}X=\{a\in\mathcal{F}(X,Y)\mid a^{\#}Y\subseteq W\}$
for some subset $W\subseteq Y$. Then 
\[
0=\mathcal{LR}=(W^{*}X)(Y^{*}U)=g(U,W)Y^{*}X
\]
 if and only if $g(U,W)=0$. And 
\[
I=\mathcal{\mathcal{R}\mathcal{L}}=(Y^{*}U)(W^{*}X)=g(X,Y)W^{*}U=W^{*}U
\]

(2) Recall $\mathfrak{fsp}(X,\Psi)=\mathfrak{u}^{*}(\mathcal{F}(X,X))=\mathfrak{su}^{*}(\mathcal{F}(X,X))=[X,X]$.
In particular $\mathfrak{fsp}(X,\Psi)$ is diagonal and $\mathcal{F}(X,X)$
is its simple associative $\mathfrak{P}^{*}$-envelope. By Corollary
\ref{LocSemi}(2) a subspace $I$ of $\mathfrak{su}^{*}(\mathcal{F}(X,X))$
is a proper inner ideal if and only if it is a regular inner ideal,
i.e. $I=\mathfrak{u}^{*}(\mathcal{RR}^{*})$ for some right ideal
$\mathcal{R}$ of $\mathcal{F}(X,X)$. As in part (1), every right
ideal is of the form $\mathcal{R=}X^{*}U=\{a\in\mathcal{F}(X,X)\mid aX\subseteq U\}$
for some subspace $U$ of $X$. Therefore $\mathcal{R}^{*}=U^{*}X=\{a\in\mathcal{F}(X,X)\mid a^{*}X\subseteq U\}$
and this is a left ideal. One has $\mathcal{R}^{*}\mathcal{R}=0$
if and only if $\Phi(U,U)=0$, i.e. $U$ is a totally isotropic subspace.
Now
\[
\begin{split}I=\mathfrak{u}^{*}(\mathcal{RR}^{*})=\mathfrak{u}^{*}((X^{*}U)(U^{*}X))=\mathfrak{u}^{*}(\Phi(X,X)U^{*}U)\\
=\mathfrak{u}^{*}(U^{*}U)=\{a-a^{*}\mid a\in U^{*}U\}=[U,U],
\end{split}
\]
as required. 

(3) Recall $\mathfrak{fo}(X,\Phi)=\mathfrak{u}^{*}(\mathcal{F}(X,X))=\mathfrak{su}^{*}(\mathcal{F}(X,X))=[X,X]$.
In particular $\mathfrak{fo}(X,\Phi)$ is diagonal and $\mathcal{F}(X,X)$
is its simple associative $\mathfrak{P}^{*}$-envelope. By Corollary
\ref{LocSemi}(2), $\mathcal{F}(X,X)$ is von Neumann regular. Then
by Proposition \ref{RLinner*}, $I$ is a regular inner ideal of $\mathfrak{fo}(X,\Phi)$
if and only if $I=u^{*}(\mathcal{RR}^{*})$ where $\mathcal{R}$ is
a right ideal of $\mathcal{F}(X,X)$ with $\mathcal{R}^{*}\mathcal{R=}0$.
As in the proof of part (2), this is equivalent to saying that $I=[U,U]$
for some totally isotropic subspace $U\mbox{\ensuremath{\subseteq}}X$.
The case of non-regular inner ideals in $\mathfrak{fo}(X,\Phi)$ is
fully considered in \cite[2.5, 3.6, 3.8]{bib:Inner} and \cite[Theorem 6.6]{bib:Benk3}.
\end{proof}
It follows from a general result, proved for nondegenerate Lie algebras
by Draper, Fern\'{a}ndez L\'{o}pez, Garc\'{i}a and G\'{o}mez Lozano,
that a simple locally finite Lie algebra contains proper minimal inner
ideals if and only if it is finitary (see \cite[Theorems 5.1 and 5.3]{bib:DLGL}).
We are going to prove a version of this result for regular inner ideals.
We will need the following facts.
\begin{prop}
\label{min}Let $A$ be a simple associative ring and let $L=[A,A].$
Then $L$ has a minimal regular inner ideal if and only if $A$ has
a proper minimal left ideal.\end{prop}
\begin{proof}
Suppose first that $A$ has a proper minimal left ideal. Since $A$
is simple with non-zero socle, by \cite[4.9]{bib:Jac}, there is a
pair $(X,Y,g)$ of dual vector spaces over a division ring $\Delta$
such that $A$ is isomorphic to the ring $\mathcal{F}(X,Y)$ of all
continuous (relative to $Y$) finite rank linear transformations of
$X$. Moreover, $\dim_{\Delta}X>1$ (otherwise $A$ is a division
ring and doesn't have proper non-zero left ideals). Take any one-dimensional
subspaces $W\subset Y$ and $V\subset X$ such that $g(V,W)=0$. Then
$I=W^{*}V$ will be a minimal regular inner ideal of $\mathfrak{fsl}(X,Y)=[\mathcal{F}(X,Y),\mathcal{F}(X,Y)]$
(see \cite[Theorem 2.5]{bib:Inner} or Theorem \ref{finitReg}(1)
above for the case $\Delta=F$). 

Suppose now that $L$ has a minimal regular inner ideal $I$. Then
$\mathcal{L}=AI$ (resp. $\mathcal{R}=IA$) is a left (resp. right)
ideal of $A$. We claim that both $\mathcal{L}$ and $\mathcal{R}$
 are non-zero. Indeed, if,
say, $AI=0$, then $IA$ is a two-sided ideal of $A$ with $(IA)^{2}=0$.
Since $A$ is simple, this implies that $IA=0$ and so $I$ is a non-zero
two-sided ideal of $A$, which is obviously a contradiction because
$I^{2}=0$. Therefore $\mathcal{L}\ne0$ and $\mathcal{R}\ne0$. Note
that $\mathcal{L}$ is a proper left ideal of $A$ (otherwise $A=AI=(AI)I=AI^{2}=0$).
We claim that $\mathcal{L}$ is a minimal left ideal of $A$. Indeed,
assume there exists a left ideal $\mathcal{L}_{1}$ of $A$ such that
$0\ne\mathcal{L}_{1}\subseteq\mathcal{L}$. By Proposition \ref{RLinner},
$I_{1}=\mathcal{R}\mathcal{L}_{1}$ is a regular inner ideal of $L$
and it is non-zero by Lemma \ref{RLgeneral}(4). Note that 
\[
I_{1}=\mathcal{R}\mathcal{L}_{1}\subseteq IAAI\subseteq I
\]
Since $I$ is minimal, $I_{1}=I$. Therefore $\mathcal{L}_{1}\supseteq A\mathcal{R}\mathcal{L}_{1}=AI_{1}=AI=\mathcal{L}$,
which is a contradiction. 
\end{proof}
A similar result holds for rings with involutions. We need the following
analogue of Lemma \ref{RLgeneral}(4). 
\begin{lem}
\label{u*LL} Let $A$ be a simple associative ring with involution
and let $\mathcal{L}$ be a non-zero left ideal of $A$ such that
$\mathcal{LL}^{*}=0$. Assume that the socle of $A$ is zero, i.e.
$A$ doesn't have minimal left ideals. Then $\mathfrak{u}^{*}(\mathcal{L}^{*}\mathcal{L})$
is non-zero. \end{lem}
\begin{proof}
Assume to the contrary that $\mathfrak{u}^{*}(\mathcal{L}^{*}\mathcal{L})=0$.
Take any non-zero $a\in\mathcal{L}$. Then $a^{*}Aa\subseteq\mathcal{L}^{*}\mathcal{L}$,
so $\mathfrak{u}^{*}(a^{*}Aa)=0$. Note that $a^{*}(x-x^{*})a\in\mathfrak{u}^{*}(a^{*}Aa)$
for all $x\in A$. Therefore $a^{*}(x-x^{*})a=0$ for all $x\in A$,
i.e. $A$ satisfies a non-trivial generalized identity with involution.
Therefore $A$ has a non-zero socle (see for example \cite[6.2.4 and 6.1.6]{BMM}),
which is a contradiction.\end{proof}
\begin{prop}
\label{min*} Let $A$ be an infinite dimensional simple associative
algebra over $F$ with involution and let $L=\mathfrak{su}^{\ast}(A)$.
Then $L$ has a minimal regular inner ideal if and only if $A$ has
a proper minimal left ideal.\end{prop}
\begin{proof}
Suppose first that $A$ has a proper minimal left ideal. Since $A$
is simple with non-zero socle, by \cite[4.9, 4.12]{bib:Jac}, $A=\mathcal{F}(X,X)$
where $X$ is a self-dual vector space over $F$ with respect to a
nondegenerate symmetric or skew-symmetric form $\Phi$ and the involution
$a\mapsto a^{*}$ of $A$ is given by $\Phi(ax,y)=\Phi(x,a^{*}y)$,
for all $x,y\in X$. Assume first that $\Phi$ is skew-symmetric.
Then $L=\mathfrak{su}^{\ast}(A)=\mathfrak{fsp}(X,\Phi)$. Take any
non-zero isotropic vector $v$ in $X$. Then 
\[
I=[Fv,Fv]=F[v,v]=F(v^{*}v+v^{*}v)=Fv^{*}v
\]
is a one-dimensional regular inner ideal by Theorem \ref{finitReg}(2).
Assume now that $\Phi$ is symmetric. Then $L=\mathfrak{su}^{\ast}(A)=\mathfrak{fo}(X,\Phi)$
Take any two-dimensional totally isotropic subspace $U$ of $X$ (this
is always possible because the ground field $F$ is algebraically
closed) and let $\{x,y\}$ be its basis. Then $I=[U,U]=F[x,y]$ is
again a one-dimensional regular inner ideal by Theorem \ref{finitReg}(3)(i).
So in both cases there exists one-dimensional (hence minimal) regular
inner ideal. 

Suppose now that $L$ has a minimal regular inner ideal $I$ and $A$
has no proper minimal left ideals. By Proposition \ref{RLinner*},
there exists a left ideal $\mathcal{L}$ of $A$ such that $\mathcal{LL}^{*}=0$
and $\mathfrak{u}^{*}(\mathcal{L}^{*}\mathcal{L})\subseteq I\subseteq\mathcal{L}^{*}\cap\mathcal{L}\cap L$.
Note that $\mathfrak{u}^{*}(\mathcal{L}^{*}\mathcal{L})$ is a non-zero
regular inner ideal by Lemma \ref{u*LL}, so $I=\mathfrak{u}^{*}(\mathcal{L}^{*}\mathcal{L})$.
Let $x\in I$ be a non-zero element. We claim that there exists a
left ideal $\mathcal{L}_{1}$ of $A$ such that $0\ne\mathcal{L}_{1}\subset\mathcal{L}$,
$x\notin\mathcal{L}_{1}$. Indeed, suppose $x$ is an element in every
non-zero left ideal contained in $\mathcal{L}$. Let $\mathcal{J}=\bigcap\{\mbox{ non-zero left ideals }\mathcal{H}\mid\mathcal{H}\subset\mathcal{L}\}$.
Then $x\in\mathcal{J}$, so $\mathcal{J}$ is non-zero. It is clear
that $\mathcal{J}$ is a minimal left ideal of $A$ giving a contradiction.
By Proposition \ref{RLinner*} and Lemma \ref{u*LL}, $I_{1}=\mathfrak{u}^{*}(\mathcal{L}_{1}^{*}\mathcal{L}_{1})$
is a non-zero regular inner ideal of $L$. Note that $I_{1}\subseteq\mathcal{L}_{1}$,
so $x\not\in I_{1}$. Therefore $I_{1}$ is properly contained in
$I$. Hence $I$ is not minimal. \end{proof}
\begin{cor}
\label{min-reg} Let $L$ be an infinite dimensional locally finite
simple Lie algebra over $F$. Then $L$ is finitary if and only if
it has a minimal regular inner ideal. \end{cor}
\begin{proof}
Suppose first that $L$ is finitary. Then by Theorem \ref{FinitaryClass},
$L=[\mathcal{F}(X,Y),\mathcal{F}(X,Y)]$ or $\mathfrak{su}^{\ast}(\mathcal{F}(X,X))$.
Both $\mathcal{F}(X,Y)$ and $\mathcal{F}(X,X)$ are infinite dimensional
and have proper minimal left ideals. Therefore by Propositions \ref{min}
and \ref{min*}, $L$ has a minimal regular inner ideal.

Suppose now that $L$ has a proper minimal regular inner ideal $I$.
Since non-diagonal Lie algebras have no proper non-zero inner ideals
(see Theorem \ref{Prop2}), $L$ must be diagonal. Therefore by \cite[Section 1]{bib:BBZ},
$L$ is either plain, i.e. $L=[A,A]$ for some simple locally finite
associative algebra $A$, or $L$ is non-plain diagonal and $L=\mathfrak{su}^{\ast}(A)$
for some simple locally finite associative algebra $A$ with involution.
By Propositions \ref{min} and \ref{min*}, $A$ has a proper minimal
left ideal. By \cite[4.9, 4.12]{bib:Jac}, $A=\mathcal{F}(X,Y)$ or
$\mathcal{F}(X,X)$, so $L$ is finitary.\end{proof}

\end{document}